\documentclass[11pt]{article}
\usepackage{amssymb}
\usepackage{mathrsfs}
\usepackage{amsfonts}
\usepackage{amsthm}
\usepackage{amsmath,amscd}
\usepackage{amsxtra}     
\usepackage{bm}
\usepackage[all,cmtip]{xy}
\usepackage{epsfig}
\usepackage{verbatim}
\usepackage{color}
\usepackage{commath}
\usepackage{mathtools}
\usepackage{cite}
\usepackage{hyperref}
\usepackage{enumerate}
\usepackage{xfrac}
\usepackage[nice]{nicefrac}
\usepackage{tikz}
\usetikzlibrary{arrows,matrix,shapes,trees}
\usepackage{cleveref}
\usepackage{blindtext}
\usepackage{geometry}
\theoremstyle{plain}
 \newtheorem{thm}{Theorem}[section]
 \newtheorem{prop}[thm]{Proposition}
 \newtheorem{lem}[thm]{Lemma}
 \newtheorem{cor}[thm]{Corollary}

 \newtheorem{lem'}{``Lemma''}
 
 \newtheorem*{claim}{Claim}
\theoremstyle{definition}
 
 \newtheorem{defn}[thm]{Definition}
\theoremstyle{remark}
 \newtheorem{rmk}[thm]{Remark}
 \numberwithin{equation}{section}

\newcommand{\Q}{{\mathbb Q}}
\newcommand{\Z}{{\mathbb Z}}
\newcommand{\C}{{\mathbb C}}
\newcommand{\F}{{\mathbb F}}

\author{\bfseries Matti W\"urthen}
\title{Vector bundles with numerically flat reduction on rigid analytic varieties and $p$-adic local systems}

\begin{document}
\maketitle
\begin{abstract}
We show how to functorially attach continuous $p$-adic representations of the profinite fundamental group to vector bundles with numerically flat reduction on a proper rigid analytic variety over $\C_p$. This generalizes results by Deninger and Werner for vector bundles on smooth algebraic varieties. Our approach uses fundamental results on the pro-\'etale site of a rigid analytic variety introduced by Scholze. This enables us to get rid of the smoothness condition and to work in the analytic category.\\
Moreover, under some mild conditions, the functor we construct gives a full embedding of the category of vector bundles with numerically flat reduction into the category of continuous $\C_p$-representations. This provides new insights into the $p$-adic Simpson correspondence in the case of a vanishing Higgs-field. 
\end{abstract}
\tableofcontents
\section{Introduction}
In complex geometry the Corlette-Simpson correspondence is a very elaborate theory relating $\C$-local systems on a compact K\"ahler manifold to semistable Higgs bundles with vanishing Chern classes (see \cite{Sim}). The case of a vanishing Higgs field goes back to a theorem of Narasimhan-Seshadri which gives an equivalence of irreducible unitary representations of the topological fundamental group and stable vector bundles of degree $0$ on a compact Riemann surface (see \cite{NS}). This was generalized by Mehta-Ramanathan to general projective manifolds (\cite{MR}) and then by Uhlenbeck-Yau to compact K\"ahler manifolds (see \cite[\S 8]{UY}). The results by Uhlenbeck-Yau give an equivalence of irreducible unitary representations and stable vector bundles $E$ satisfying $\int \omega^{dimX-1}\wedge c_{1}(E)=\int \omega^{dimX-2}\wedge c_{2}(E)=0$, where $\omega$ is a K\"ahler-class.\\
For $p$-adic varieties Faltings in \cite{Fa} has proposed a theory relating (small) Higgs bundles and so called (small) generalized representations (see in particular \cite[Theorem 6]{Fa}). This is expanded in the work \cite{AGT}. Moreover, in \cite{LZ}, using the methods from \cite{Sch1} and \cite{KLiu}, Liu and Zhu have constructed a functor from the category of $\Q_p$-local systems on a smooth rigid analytic variety $X$ over a finite extension $K$ of $\Q_p$, to the category of Higgs bundles on $X_{\hat{\overline{K}}}$ (see \cite[Theorem 2.1]{LZ}). What is still missing from these approaches is the other direction. In Faltings's approach the problem is to determine the Higgs-bundles whose associated generalized representations come from actual representations. One may expect that, similarly to the complex situation, this should be related to a semistability condition on the Higgs bundle. In the zero Higgs field case, such a semistability condition was found by Deninger and Werner (cf. \cite{DW1}, \cite{DW3}) for vector bundles on proper smooth algebraic varieties over $\bar{\Q}_p$. Namely, they construct a functor from the category of vector bundles, which possess an integral model whose special fiber is a numerically flat vector bundle, to the category of continuous $\C_p$-representations of the \'etale fundamental group. We remark here that on a complex compact K\"ahler manifold numerically flat bundles are precisely the semistable bundles satisfying $\int \omega^{dimX-1}\wedge c_{1}(E)=\int \omega^{dimX-2}\wedge c_{2}(E)=0$ (see \cite[Theorem 1.18]{DPS}). So the condition one is familiar with from complex geometry shows up here as a condition on the special fiber of a model.\\
The main goal of this article is to develop a new approach to the Deninger-Werner correspondence via the pro-\'etale site introduced by Scholze in \cite{Sch1}. Using this approach we can get rid of many assumptions in the results of \cite{DW3}. In particular we do not need any smoothness of the variety $X$. Also, we don't need to assume that it is defined over a finite extension of $\Q_{p}$. Moreover our construction now takes place in the analytic category, so that we can assume $X$ to be any proper connected rigid analytic variety over $\C_p$. We consider vector bundles $E$ on $X$, for which there exists a proper flat formal scheme $\mathcal{X}$ with generic fiber $X$ over $Spf(\mathcal{O}_{\C_p})$ and a vector bundle $\mathcal{E}$ on $\mathcal{X}$ with generic fiber $E$, such that $\mathcal{E}\otimes \overline{\mathbb{F}}_{p}$ is a numerically flat vector bundle. These kinds of vector bundles form a tensor category $\mathcal{B}^{s}(X)$.  Our main result is the following
\begin{thm}
Let $X$ be a proper connected seminormal rigid analytic variety over $\C_p$. Then there is a fully faithful functor
\begin{center}
$
DW:\mathcal{B}^{s}(X)\to Rep_{\pi_1(X)}(\C_p)
$
\end{center}
which is exact and compatible with tensor products, duals, internal homs and exterior products. 
\end{thm}
Here $Rep_{\pi_1(X)}(\C_p)$ denotes the category of continuous representations of the \'etale fundamental group $\pi_{1}(X)=\pi_{1}^{\acute{e}t}(X, x)$ (for a fixed base point $x$) on finite dimensional $\C_p$-vector spaces.\\
Using the methods developed in \cite{Sch1} we are able to work for the most part on the generic fiber, by which we can avoid the complications arising in \cite{DW3} in the study of integral models. At the same time the results from \cite{Sch1} and \cite{KLiu} allow us to derive the full faithfulness of the Deninger-Werner functor, which could not be seen from the construction in \cite{DW3}. We note that this last point is close in spirit to the article \cite{Xu}, where the constructions by Deninger and Werner are analyzed via the Faltings topos and full faithfulness is established in the curve case.
\\[0.2in]
We want to remark that one of the main open problems is still to find out which vector bundles admit a model with numerically flat reduction. We have essentially nothing new to add to this, but wish to draw the reader's attention to the recent preprint \cite{HW} which reports on progress on this problem.\\
\\[0.2in]
Let us sketch how we go about proving the above theorem. The category $Rep_{\pi_1(X)}(\mathcal{O}_{\C_p})$ of continuous representations on finite free $\mathcal{O}_{\C_p}$-modules is equivalent the category of locally free $\hat{\mathcal{O}}_{\C_p}$-sheaves (i.e. local systems with coefficients in $\mathcal{O}_{\C_p}$, see Definition \ref{defn2}). There is a fully faithful functor from the latter category to the category of locally free $\hat{\mathcal{O}}_{X}^{+}$-modules, where $\hat{\mathcal{O}}_{X}^{+}$ is the completed integral structure sheaf on the pro-\'etale site of $X$, given by $\mathbb{L}\mapsto \mathbb{L}\otimes \hat{\mathcal{O}}_{X}^{+}$.\footnote{We remark that locally free (almost) $\hat{\mathcal{O}}_{X}^{+}$-modules are analogous to the objects called generalized representations studied in \cite{Fa} and \cite{AGT}.} We note here that the analogous statement on the Faltings topos is also the starting point of Faltings's $p$-adic Simpson correspondence.
\\
One can then show (see Corollary \ref{full}) that the essential image of the above functor is given by the $\hat{\mathcal{O}}_{X}^{+}$-modules which become trivial on a profinite \'etale covering of $X$. If $\mathcal{E}$ is a vector bundle on a formal scheme $\mathcal{X}$ over $Spf(\mathcal{O}_{\C_p})$, with generic fiber $X$, we can form its pullback $\mathcal{E}^{+}$ to the pro-\'etale site of $X$. Its $p$-adic completion $\hat{\mathcal{E}}^{+}$ is an $\hat{\mathcal{O}}^{+}_{X}$-module. We will then show the following:
\begin{thm}
Let $\mathcal{X}$ be a proper flat connected formal scheme over $Spf(\mathcal{O}_{\C_p})$ and $\mathcal{E}$ a vector bundle on $\mathcal{X}$ with numerically flat reduction. Then $\hat{\mathcal{E}}^{+}$ is trivialized by a profinite \'etale cover.
\end{thm}
In the terminology of \cite{Xu} (see Definition \ref{wt}) this proves that all vector bundles with numerically flat reduction are Weil-Tate. This was shown in loc. cit. for the case of curves using the constructions from \cite{DW1}.\\ 
This theorem will also be used to construct \'etale parallel transport on $\mathcal{E}$, as in \cite{DW3}.\\
The proof of the theorem follows very much the path laid out in \cite{DW3}. In particular one shows that a vector bundle $\mathcal{E}$ with numerically flat reduction is trivialized modulo $p$ after pullback along a composition of a finite \'etale cover and some power of the absolute Frobenius map. One is then faced with two problems: One is dealing with the Frobenius pullback and the other is to inductively get rid of obstructions preventing the bundle in question to be trivial modulo $p^{n}$. Both problems become much simpler after pulling back to the pro-\'etale site (Theorem \ref{tf}).
\\
Let us make some remarks on the contents of the individual sections. Section 2 is a recollection of the results from \cite{Sch1} which are needed in this article. In section 3 we first show how to attach continuous $\mathcal{O}_{\textbf{C}}$-representations to integral vector bundles on the pro-\'etale site whose $p$-adic completion is trivialized by a finite \'etale cover (Theorem \ref{thm21}). These constructions work for proper rigid analytic varieties over any complete algebraically closed perfectoid field $\textbf{C}$. We then specialize to the case $\textbf{C}=\C_p$ where we have the alternative viewpoint of local systems. Finally, we show that any $\mathcal{O}^{+}$-module which is trivialized modulo $p$ by a Frobenius pullback will give rise to a representation (Theorem \ref{tf}).\\
Section 4 then mostly deals with vector bundles with numerically flat reduction. We first generalize the results from \cite[\S 2]{DW3} on numerically flat vector bundles on projective schemes over finite fields to the non-projective case (Theorem \ref{thm31}). The proof of this is an application of $v$-descent for vector bundles on perfect schemes as established in \cite{BS}. Then the results from section 3 are used to construct the Deninger-Werner functor for vector bundles with numerically flat reduction. We then show that the discussion from section 3 can be improved to construct an \'etale parallel transport functor for the given bundle and that our construction indeed recovers the results from \cite{DW3} (Theorem \ref{dwthm}). \\
We finish by showing that the cohomology of the constructed local systems come with a Hodge-Tate spectral sequence. Note that this has also been treated in \cite{Xu} via the Faltings topos building on the work of Abbes, Gros and Tsuji (see in particular \cite[Proposition 11.7]{Xu} and \cite[Proposition 11.8]{Xu}).
\section*{Acknowledgements}
I would like to thank Annette Werner for showing interest in my work and for making many remarks and suggestions on an earlier version of this paper.\\
I wish to thank Shizhang Li for pointing out an error in an earlier version of this paper and for many fruitful exchanges.\\
I would like to thank Bhargav Bhatt for drawing my attention to \cite{guo} and suggesting that the results in loc. cit. should yield the degeneration of the Hodge-Tate spectral sequence.\\
Lastly, I wish to thank Lucas Mann, Peter Scholze and Adrian Zorbach for discussions.\\
This work is part of my Phd-thesis at the University of Duisburg-Essen and was financially supported by the SFB/TR 45 ''Periods, Moduli Spaces and Arithmetic of Algebraic Varieties'' of the DFG (German Research Foundation).

\section{Preliminaries}
\subsection{Notation}
Throughout this paper we will work with proper rigid analytic varieties (viewed as a full subcategory of the category of adic spaces (see \cite{Hub})) over an algebraically closed perfectoid field $\textbf{C}$ of characteristic $0$. We denote by $\mathcal{O}_{\textbf{C}}\subset \textbf{C}$ its ring of integers. If $\Gamma$ is the value group of $\textbf{C}$, we denote by $log\Gamma\subset \mathbb{R}$ the induced subset obtained by taking the logarithm with base $\vert p\vert$. Then for any $\epsilon\in log\Gamma$ we choose an element $p^{\epsilon}\in \textbf{C}$, which satisfies $\vert p^{\epsilon}\vert=\vert p \vert^{\epsilon}$.\\
Moreover we fix a pseudouniformizer $t$ in the tilt $\mathcal{O}_{\textbf{C}^{\flat}}$ such that $t^{\sharp}=p$.\\
Whenever we speak about almost mathematics we mean almost mathematics with respect to the maximal ideal $\mathfrak{m}\subset \mathcal{O}_{\textbf{C}}$.\\
If $X$ is a rigid analytic space over a non-archimedean field $K$, then by formal model we mean an admissible formal scheme over $\mathcal{O}_{K}$ with generic fiber $X$. 
\subsection{The pro-\'etale site of a rigid analytic space}
Let $X$ be a locally noetherian adic space over $Spa(\Q_p, \Z_p)$. The pro-\'etale site $X_{pro\acute{e}t}$ of $X$ was introduced in \cite{Sch1}. The idea is that one wants to extend the usual \'etale site to allow inverse limits along finite \'etale morphisms. Since inverse limits along affinoid morphism may not be well behaved in the category of adic spaces one simply considers formal filtered pro-systems $\varprojlim_{i\in I}Y_{i}\to X$ of \'etale maps, such that there exists some $i_{0}$ such that the transition maps $Y_{i}\to Y_{i'}$ are all finite \'etale for $i'\geq i_{0}$. We refer to \cite{BMS} \S 5 for the precise definitions.
\begin{rmk}
There is also a newer version of the pro-\'etale site. See for example \cite[Definition 3.1.]{leb} for a definition for analytic adic spaces. We will only use the original version of the pro-\'etale site, as this simplifies the exposition in our case. The version we use is now also sometimes referred to as the flattening pro-\'etale site.
\end{rmk}
Every \'etale morphism is pro-\'etale, which gives a canonical projection 
\begin{center}
$\nu:X_{pro\acute{e}t}\to X_{\acute{e}t}$.
\end{center}
There is also a natural projection of sites $\lambda:X_{pro\acute{e}t}\to X_{an}$, where $X_{an}$ denotes the analytic topology of open subsets of $X$, given by the composition 
\begin{center}
$X_{pro\acute{e}t}\xrightarrow{\nu}X_{\acute{e}t}\to X_{an}$. 
\end{center}
\begin{defn}\cite[Definition 4.1]{Sch1}\label{defn1}\\
We have the following structure sheaves\\
\begin{itemize}
\item $\mathcal{O}_{X}^{+}:=\nu^{-1}\mathcal{O}^{+}_{X_{\acute{e}t}}$,
$\mathcal{O}_{X}=\nu^{-1}\mathcal{O}_{X_{\acute{e}t}}$
\\
\item $\hat{\mathcal{O}}_{X}^{+}=\varprojlim_n \mathcal{O}_{X}^{+}/p^{n}$, 
$\hat{\mathcal{O}}_{X}=\hat{\mathcal{O}}_{X}^{+}[\frac{1}{p}]$  (completed structure sheaves)
\\
\item $\hat{\mathcal{O}}_{X^\flat}^{+}=\varprojlim_{\Phi}\mathcal{O}_{X}^{+}/p$ (tilted structure sheaf)
\end{itemize}
where $\phi$ denotes the (surjective) Frobenius on $\mathcal{O}^{+}_{X}/p$.
\end{defn}
\begin{rmk}
Let $f:X\to Y$ be a morphism of adic spaces over $Spa(\Q_p, \Z_p)$. Then the map $f_{pro\acute{e}t}^{-1}\mathcal{O}^{+}_{Y}\to \mathcal{O}^{+}_{X}$ extends to the $p$-adic completion, so we get a map $f_{pro\acute{e}t}^{-1}\hat{\mathcal{O}}_{Y}^{+}\to \hat{\mathcal{O}}_{X}^{+}$. Hence we can define the pullback of a finite locally free $\hat{\mathcal{O}}^{+}_{Y}$-module $\mathcal{F}$ along $f$ in the usual way.
\end{rmk}
Now let $\mathcal{X}$ be an admissible formal scheme over $Spf(\mathcal{O}_{\textbf{C}})$. By \cite[Proposition 1.9.1]{Hu} there is an adic space $X$ over $Spa(\textbf{C}, \mathcal{O}_{\textbf{C}})$, which comes with a specialization map\\$sp:(X, \mathcal{O}_{X_{an}}^{+})\to (\mathcal{X}, \mathcal{O}_{\mathcal{X}})$ of locally ringed spaces, and is such that for any morphism of locally ringed spaces $f:(Z, \mathcal{O}_{Z_{an}}^{+})\to (\mathcal{X}, \mathcal{O}_{\mathcal{X}})$ where $Z$ is an adic space over $Spa(\textbf{C}, \mathcal{O}_{\textbf{C}})$, there is a unique morphism $g:Z\to X$ such that $f=sp \circ g$. Moreover, $X$ is canonically isomorphic to the adic space associated to Raynaud's rigid analytic generic fiber of $\mathcal{X}$. We call $X$ the generic fiber of $\mathcal{X}$.\\
We again have a canonical projection $\mu=sp \circ \lambda:X_{pro\acute{e}t}\to \mathcal{X}_{Zar}$. There is a natural map $\mu^{-1}\mathcal{O}_{\mathcal{X}}\to \mathcal{O}_{X}^{+}$. Hence for any $\mathcal{O}_{\mathcal{X}}$-module $\mathcal{E}$, we can define the associated $\mathcal{O}_{X}^{+}$-module $\mathcal{E}^{+}:=\mu^{-1}\mathcal{E}\otimes_{\mu^{-1}\mathcal{O}_{\mathcal{X}}}\mathcal{O}^{+}_{X}$. One easily proves the following:
\begin{lem}\label{lem11}
Let $f:\mathcal{X}\to \mathcal{Y}$ be a morphism of admissible formal schemes over $Spf(\mathcal{O}_{\textbf{C}})$. Then for any $\mathcal{O}_{\mathcal{Y}}$-module $\mathcal{E}$ there is a canonical isomorphism $(f_{\textbf{C}})_{pro\acute{e}t}^{*}(\mathcal{E}^{+})\cong (f^{*}\mathcal{E})^{+}$.
\end{lem}
The main result from \cite{Sch1} that we need is the so called primitive comparison theorem:
\begin{thm}\cite[Theorem 5.1]{Sch1}\label{cisoo}
Let $X$ be a proper rigid analytic space over $Spa(\textbf{C}, \mathcal{O}_{\textbf{C}})$, where $\textbf{C}$ is an algebraically closed perfectoid field of characteristic $0$ with ring of integers $\mathcal{O}_{\textbf{C}}$.
Then the canonical maps
\\
\begin{center}
$H^{i}(X_{et}, \mathbb{Z}_{p}/p^{n})\otimes \mathcal{O}_{\textbf{C}}/p^{n}\to H^{i}(X, \mathcal{O}_{X}^{+}/p^{n})$
\\[0.1in]
$H^{i}(X_{et}, \mathbb{Z}_{p})\otimes \mathcal{O}_{\textbf{C}}\to H^{i}(X, \hat{\mathcal{O}}_{X}^{+})$
\\
\end{center}
are almost isomorphisms for all $i\geq 0$.\\
If $X$ is in addition smooth or the analytification of a proper scheme over $Spec(\textbf{C})$, then the canonical maps
\\
\begin{center}
$H^{i}(X_{et}, \mathbb{L}/p^{n})\otimes \mathcal{O}_{\textbf{C}}/p^{n}\to H^{i}(X, \mathcal{O}_{X}^{+}/p^{n}\otimes \mathbb{L}/p^{n})$
\\[0.1in]
$H^{i}(X_{et}, \mathbb{L})\otimes \mathcal{O}_{\textbf{C}}\to H^{i}(X, \mathbb{L}\otimes \hat{\mathcal{O}}_{X}^{+})$
\\
\end{center}
are almost isomorphisms for any lisse $\mathbb{Z}_{p}$-sheaf $\mathbb{L}$.
\end{thm}
\begin{proof}
The mod $p$ statements can be found in \cite[Theorem 5.1]{Sch1} (smooth case), \cite[Theorem 3.13]{Sch2} (algebraic case) and \cite[Theorem 3.17]{Sch2} (for general proper rigid analytic varieties). The full statements all then follow by induction combined with \cite[Lemma 3.18]{Sch1}.
\end{proof}
The categories of locally free sheaves on the pro-\'etale, \'etale and analytic sites all agree:
\begin{lem}\cite[Lemma 7.3]{Sch1}\label{veceq}
Pullback along the natural projections 
\begin{center}
$X_{pro\acute{e}t}\xrightarrow{\nu}X_{\acute{e}t}\to X_{an}$
\end{center}
induces equivalences of categories between the categories of finite locally free modules over $\mathcal{O}_{X}$, resp. $\mathcal{O}_{X_{\acute{e}t}}$, resp. $\mathcal{O}_{X_{an}}$.
\end{lem}
Note that in \cite{Sch1} $X$ is assumed to be smooth, but the proof given there works also for the general case (see also \cite[Theorem 8.2.22]{KLiu1}).\\
By a vector bundle on $X$ we usually mean a locally free $\mathcal{O}_{X_{an}}$-module. We will however often use the above lemma to freely switch between the topologies. We hope that this will not be a source of confusion.
\begin{rmk}
If $X$ is quasi-compact and quasi-separated, pullback along $\nu$ also induces an equivalence of categories between finite locally free $\mathcal{O}^{+}_{X_{\acute{e}t}}$-modules and finite locally free modules over $\mathcal{O}_{X}^{+}$: As $\nu_{*}\mathcal{O}_{X}^{+}=\mathcal{O}_{X_{\acute{e}t}}^{+}$ (by \cite[Corollary 3.17 (i)]{Sch1}) one sees that $\nu^{*}$ is fully faithful. Now assume $\mathcal{E}^{+}$ is a finite locally free $\mathcal{O}^{+}_{X}$-module given by a gluing datum on $\tilde{Y}\times_{X} \tilde{Y}$ for some pro-\'etale cover $\tilde{Y}\to X$, which we can assume to be qcqs. By \cite[Lemma 3.16]{Sch1} one has $\mathcal{O}^{+}_{X}(V)=\varinjlim \mathcal{O}^{+}_{X}(V_{j})$ for any qcqs $V\in X_{pro\acute{e}t}$. So the gluing datum descends to some $Y_{i}\times_{X} Y_{i}$, which shows that $\mathcal{E}^{+}$ lies in the essential image of $\nu^{*}$.
\end{rmk}
\subsection{Local systems with coefficients in $\mathcal{O}_{\mathbb{C}_{p}}$}
Let $X$ be a connected locally noetherian adic space which is quasi-compact and quasi-separated. We fix a geometric point $\bar{x}$ of $X$ and denote by $\pi_{1}(X):=\pi^{\acute{e}t}_{1}(X, \bar{x})$ the profinite fundamental group with respect to the base point $\bar{x}$. Then by \cite[Proposition 3.5]{Sch1} the category of profinite \'etale covers of $X$ is equivalent to the category of profinite sets with a continuous $\pi_{1}(X)$-action. For any topological ring $R$ we denote by $Rep_{\pi_{1}(X)}(R)$ the category of continuous representations of $\pi_{1}(X)$ on finite rank free $R$-modules.
\begin{rmk}
A complete treatment of the fact that the category of finite \'etale covers of a connected locally noetherian adic space forms a Galois category can be found in \cite[\S 4]{dllz} (where the authors even treat the case of adic spaces equipped with a non-trivial log structure).
\end{rmk}
\begin{defn}[cf. \cite{Sch1} 8.1]\label{defn2}
We define $\hat{\mathcal{O}}_{\mathbb{C}_p}:=\varprojlim_{n}\mathcal{O}_{\mathbb{C}_p}/p^{n}$, where $\mathcal{O}_{\mathbb{C}_p}/p^{n}$ denotes the constant sheaf on $X_{pro\acute{e}t}$. An $\mathcal{O}_{\mathbb{C}_p}$-local system is a finite locally free $\hat{\mathcal{O}}_{\mathbb{C}_p}$-module.\\
We denote the category of $\hat{\mathcal{O}}_{\C_p}$-local systems by $LS_{\mathcal{O}_{\C_p}}(X)$.\\
Similarly an $\mathcal{O}_{\C_p}/p^{n}$-local system is a finite locally free module over the constant sheaf $\mathcal{O}_{\C_p}/p^{n}$ on $X_{pro\acute{e}t}$. We write $LS_{\mathcal{O}_{\C_p}/p^{n}}(X)$ for the category of $\mathcal{O}_{\C_p}/p^{n}$-local systems.
\end{defn}
\begin{rmk}\label{remake}
\begin{itemize}
\item As $X$ is qcqs one has $LS_{\mathcal{O}_{\C_p}/p^{n}}(X)=colim_{K/\Q_p}LS_{\mathcal{O}_{K}/p^{n}}(X)$, where $K$ runs over finite extensions of $\Q_p$. This is beacause $\mathcal{O}_{\C_p}/p^{n}=\varinjlim_{K/\Q_p} \mathcal{O}_{K}/p^{n}$.
\item $LS_{\mathcal{O}_{\C_p}/p^{n}}(X)$ is equivalent to the category of finite locally free $\mathcal{O}_{\C_p}/p^{n}$-modules on the \'etale site, i.e. every $\mathbb{L}\in LS_{\mathcal{O}_{\C_p}/p^{n}}(X)$ is in the essential image of $\nu^{*}$, where $\nu:X_{pro\acute{e}t}\to X_{\acute{e}t}$ is again the natural projection. For this note that if $A$ is any abelian group and $A_{pro\acute{e}t}$, $A_{\acute{e}t}$ the constant sheaves with values in $A$ on $X_{pro\acute{e}t}$, resp. on $X_{\acute{e}t}$, we get $\nu_{*}A_{pro\acute{e}t}=\nu_{*}\nu^{*}A_{\acute{e}t}=A_{\acute{e}t}$ by \cite[Corollary 3.17]{Sch1}.
\item Since $\mathcal{O}_{K}\subset \mathcal{O}_{\C_p}$ is flat, for any finite extension $K/\Q_p$, for any $\mathcal{O}_{K}/p^{n}$-local system $\mathbb{L}$ we have $H^{i}(X_{pro\acute{e}t}, \mathbb{L}\otimes \mathcal{O}_{\C_p}/p^{n})=H^{i}(X_{pro\acute{e}t}, \mathbb{L}) \otimes \mathcal{O}_{\C_p}/p^{n}$, for all $i\geq 0$.\\
\end{itemize}
\end{rmk}
Consider the category $LS_{\mathcal{O}_{\C_p \bullet}}(X)$ of inverse systems $\{ \mathbb{L}_{n} \}$ of finite free $\mathcal{O}_{\C_p}/p^{n}$-modules on $X_{pro\acute{e}t}$, where $\{ \mathbb{L}_{n} \}$ is isomorphic to an inverse system $\{ \mathbb{L}'_{n} \}$, satisfying $\mathbb{L}'_{n+1}/p^{n}\cong \mathbb{L}'_{n}$. Then there is a functor $LS_{\mathcal{O}_{\C_p \bullet}}(X)\to LS_{\mathcal{O}_{\C_p}}(X)$, taking $\{ \mathbb{L}_{n} \}$ to $\varprojlim \mathbb{L}_{n}$. The proof of \cite[Theorem 4.9]{Sch1} shows that the inverse system $\{ \mathbb{L}_{n} \}$ satisfies the conditions from \cite[Lemma 3.18]{Sch1}, which gives the following:
\begin{prop}[cf. \cite{Sch1} Proposition 8.2]\label{prop82}
The functor $LS_{\mathcal{O}_{\C_p \bullet}}(X)\to LS_{\mathcal{O}_{\C_p}}(X)$ is an equivalence of categories.
\end{prop}
\begin{prop}\label{repeq}
There is an equivalence of categories
\begin{center}
$
LS_{\mathcal{O}_{\C_p}}(X)
\leftrightarrow
Rep_{\pi^{\acute{e}t}_{1}(X, \bar{x})}(\mathcal{O}_{\C_p})$
\end{center}
\end{prop}
\begin{proof}
The following arguments are well known. First fix $n\geq 1$. As usual, finite $\pi_{1}(X)$-sets correspond to finite \'etale covers, and $\mathbb{L}\mapsto \mathbb{L}_{\bar{x}}$ gives an equivalence of categories 
\begin{center}
$LS_{\mathcal{O}_{K}/p^{n}}(X)\leftrightarrow Rep_{\pi_{1}^{\acute{e}t}(X, \bar{x})}(\mathcal{O}_{K}/p^{n})$
\end{center}
for all finite extensions $K/\Q_p$. Clearly, these equivalences are compatible with base extensions $\mathcal{O}_{K}/p^{n}\to \mathcal{O}_{K'}/p^{n}$, for $K\subset K'$. So we get 
\begin{center}
$colim_{K\subset \overline{\Q}_{p}}LS_{\mathcal{O}_{K}/p^{n}}(X)\leftrightarrow colim_{K\subset \overline{\Q}_{p}}Rep_{\pi_{1}^{\acute{e}t}(X, \bar{x})}(\mathcal{O}_{K}/p^{n})$.
\end{center}
Now by Remark \ref{remake}, we know that $colim_{K\subset \overline{\Q}_{p}}LS_{\mathcal{O}_{K}/p^{n}}$ is equivalent to $LS_{\mathcal{O}_{\C_p}/p^{n}}(X)$. On the other hand, that $colim_{K\subset \overline{\Q}_{p}}Rep_{\pi_{1}^{\acute{e}t}(X, \bar{x})}(\mathcal{O}_{K}/p^{n})$ is equivalent to $Rep_{\pi_{1}^{\acute{e}t}(X, \bar{x})}(\mathcal{O}_{\C_p}/p^{n})$ follows from the fact that $\pi_{1}(X)$ is compact: since $GL_{n}(\mathcal{O}_{\C_p}/p^{n})$ carries the discrete topology, the image $\rho(\pi_{1}(X))$ will be finite for any continuous representation $\rho:\pi_{1}(X)\to GL_{n}(\mathcal{O}_{\C_p}/p^{n})$.\\
Now passing to the $p$-adic completion, using Proposition \ref{prop82}, gives the claim.
\end{proof}
One can then also generalize the primitive comparison theorem to the case of $\mathcal{O}_{\C_p}$-coefficients.
\begin{thm}\label{ciso}
Let $X$ be a proper smooth rigid analytic space over $Spa(\C_p, \mathcal{O}_{\C_p})$. And let $\mathbb{L}$ be an $\hat{\mathcal{O}}_{\mathbb{C}_p}$-local system on $X$. The canonical map 
\begin{center}
$
H^{i}(X_{pro\acute{e}t}, \mathbb{L})\to H^{i}(X_{pro\acute{e}t}, 
\mathbb{L}\otimes_{\hat{\mathcal{O}}_{\mathbb{C}_p}}\hat{\mathcal{O}}_{X}^{+})
$
\end{center}
is an almost isomorphism, for all $i\geq 0$.
\end{thm}
\begin{proof}
By the above remark we see that $\mathbb{L}/p\cong \mathbb{L}'\otimes \mathcal{O}_{\C_p}/p$ where $\mathbb{L}'$ is defined over $\mathcal{O}_{K}/p$ for $K/\Q_p$ a finite extension. Let $\pi$ be a uniformizer of $K$. Then $\mathbb{L}'/\pi$ is an $\F_q$-local system, for some $q=p^{m}$. Now, if we replace the Frobenius occuring in the proof of Theorem 5.1 in \cite{Sch1} everywhere by its $m$-th power $x\mapsto x^{q}$, the proof goes through for $\mathbb{F}_{q}$-local systems and we get an almost isomorphism $H^{i}(X_{\acute{e}t}, \mathbb{L}'/\pi)\otimes \mathcal{O}_{\C_p}/p\cong^{a}H^{i}(X_{\acute{e}t}, \mathbb{L}'/\pi\otimes \mathcal{O}^{+}/p)$. But then by induction along the exact sequences 
\begin{center}
$0\rightarrow \pi^{n-1}\mathbb{L}'/\pi^{n}\rightarrow \mathbb{L}'/\pi^{n}\rightarrow \mathbb{L}'/\pi^{n-1}\rightarrow 0$
\end{center}
we find that 
\begin{center}
$H^{i}(\mathbb{L}/p)=H^{i}(\mathbb{L}')\otimes_{\mathcal{O}_{K}/p}\mathcal{O}_{\C_p}/p\to H^{i}(\mathbb{L}'\otimes_{\mathcal{O}_{K}/p}\hat{\mathcal{O}}^{+}/p)=H^{i}(\mathbb{L}\otimes_{\hat{\mathcal{O}}_{\C_p}}\hat{\mathcal{O}}^{+}/p)$
\end{center}
is an almost isomorphism. But then the full statement follows again by induction and using \cite[Lemma 3.18]{Sch1} as in the case of $\hat{\Z}_p$-local systems (see the proof of \cite[Theorem 8.4]{Sch1}).
\end{proof}
\begin{rmk}
There is a functor 
\begin{center}
$
LS_{\mathcal{O}_{\C_p}}(X) \to LF(\hat{\mathcal{O}}^{+}_{X}) 
$
\end{center}
which takes $\mathbb{L}$ to $\mathbb{L}\otimes \hat{\mathcal{O}}^{+}_{X}$. Assume now that $X$ is connected and proper smooth over $Spa(\C_p, \mathcal{O}_{\C_p})$. Then one immediately gets from Theorem \ref{ciso} that the induced functor
\begin{center}
$Rep_{\pi_{1}(X)}(\mathcal{O}_{\C_p})\otimes \Q \cong LS_{\mathcal{O}_{\C_p}}(X) \otimes \Q\to LF(\hat{\mathcal{O}}_{X})$, 
\end{center}
taking $\rho$ to $\mathbb{L}_{\rho}\otimes \hat{\mathcal{O}}_{X}$, is fully faithful. Here for a representation $\rho$, we denote by $\mathbb{L}_{\rho}$ the associated local system.\\
At the integral level Theorem \ref{ciso} shows that one has a fully faithful embedding of $\mathcal{O}_{\C_p}$-local systems into the category of finite locally free almost $\hat{\mathcal{O}}_{X}^{+}$-modules. Note however, that the discussions in the following section will show that full faithfulness holds also at the integral level (without passing to almost modules) and without any restrictions on $X$ (see Corollary \ref{full}).
\end{rmk}
\section{Representations attached to $\mathcal{O}_{X}^{+}$-modules}
\subsection{Trivializable $\hat{\mathcal{O}}_{X}^{+}$-modules and representations}
Fix a $Spa(\textbf{C}, \mathcal{O}_{\textbf{C}})$-valued point $x$ of $X$. We will show how to attach a continuous $\mathcal{O}_{\textbf{C}}$-representation of $\pi_{1}(X):=\pi^{\acute{e}t}_{1}(X, x)$ to certain $\mathcal{O}_{X}^{+}$-modules $\mathcal{E}^{+}$ for which the $p$-adic completion $\hat{\mathcal{E}}^{+}$ is trivialized on a profinite \'etale cover, i.e. an inverse limit of finite \'etale surjective maps.
\\
Let $X$ be a proper connected rigid analytic space over $Spa(\textbf{C}, \mathcal{O}_{\textbf{C}})$. Then the only global functions are the constant ones, i.e. $\Gamma(X, \mathcal{O}_{X})=\textbf{C}$. As $\Gamma(X, \mathcal{O}_{X}^{+})$ consists of the functions $f$ for which $\abs{f(x)}\leq 1$ for all $x\in X$, we see that $\Gamma(X, \mathcal{O}_{X}^{+})=\mathcal{O}_{\textbf{C}}$. Similarly one has $\Gamma(X, \hat{\mathcal{O}}_{X}^{+})\subset \Gamma(X, \hat{\mathcal{O}}_{X})=\textbf{C}$. And hence $\Gamma(X, \hat{\mathcal{O}}_{X}^{+})=\mathcal{O}_{\textbf{C}}$.\\
We first record the following:
\begin{lem}\label{lem31}
Let $\tilde{Y}=\varprojlim_{i}Y_{i}\to X$ be a profinite \'etale cover and let $\mathcal{E}^{+}$ be a locally free $\mathcal{O}_{X}^{+}$-module, such that $\hat{\mathcal{E}}^{+}\vert_{\tilde{Y}}$ is trivial. Then for any $n\geq 1$ there exists some $i$, such that $\mathcal{E}^{+}/p^{n}$ becomes trivial on $Y_{i}$.
\end{lem}
\begin{proof}
Let $\nu:X_{pro\acute{e}t}\to X_{\acute{e}t}$ denote the canonical projection. There exists a locally free $\mathcal{O}^{+}_{X_{\acute{e}t}}$-module $\mathcal{F}$, such that $\nu^{*}\mathcal{F}=\mathcal{E}^{+}$. Hence we also have $\hat{\mathcal{E}}^{+}/p^{n}=\nu^{*}(\mathcal{F}/p^{n})$. But then by \cite[Lemma 3.16]{Sch1} $\hat{\mathcal{E}}^{+}/p^{n}$ is the sheaf given by $\hat{\mathcal{E}}^{+}/p^{n}(V)=\varinjlim_{j}\mathcal{F}/p^{n}(V_{j})$ for any qcqs object $V=\varprojlim_{j}V_{j}\in X_{pro\acute{e}t}$.\\
As all $Y_{i}$ are quasi-compact and quasi-separated we have that $\tilde{Y}$ is qcqs by \cite[Lemma3.12]{Sch1} (v). Now note that if $\mathcal{E}, \mathcal{E}'$ are locally free $\mathcal{O}^{+}_{X}/p^{n}$-modules, then by what we said above 
\begin{center}
$Hom(\mathcal{E}\vert_{\tilde{Y}}, \mathcal{E}'\vert_{\tilde{Y}})=\mathcal{H}om(\mathcal{E}, \mathcal{E}')(\tilde{Y})=\varinjlim \mathcal{H}om(\mathcal{E}, \mathcal{E}')(Y_{j})=\varinjlim Hom(\mathcal{E}\vert_{Y_{j}}, \mathcal{E}'\vert_{Y_{j}})$.
\end{center}
From this we see that the isomorphism $(\hat{\mathcal{E}}^{+}/p^{n})\vert_{\tilde{Y}}\cong (\hat{\mathcal{O}}_{\tilde{Y}}^{+}/p^{n})^{r}$ descends to an isomorphism $(\hat{\mathcal{E}}^{+}/p^{n})\vert_{Y_i}\cong \hat{\mathcal{O}}^{+}_{Y_i}/p^{n}$ for some large enough $i$.
\end{proof}
\begin{lem}
Let $X$ be proper connected and $\tilde{Y}=\varprojlim_{i}Y_{i}\to X$ be a profinite \'etale cover where each $Y_{i}$ is connected. Then $\Gamma(\tilde{Y}, \hat{\mathcal{O}}_{X}^{+})=\mathcal{O}_{\textbf{C}}$
\end{lem}
\begin{proof}
Let $\mathcal{O}^{+a}_{X}$ be the almost version of the integral structure sheaf on $X_{pro\acute{e}t}$. Then Theorem \ref{cisoo} gives $\mathcal{O}^{+a}_{X}/p^{n}(\tilde{Y})=\varinjlim_{i}\mathcal{O}^{+a}/p^{n}(Y_{i})=\mathcal{O}_{\textbf{C}}^{a}/p^{n}$ (As the $Y_{i}$ are all connected the direct limit is taken along isomorphisms). As the inverse limit of sheaves agrees with the inverse limit of presheaves, we get $\hat{\mathcal{O}}^{+a}_{X}(\tilde{Y})=\mathcal{O}_{\textbf{C}}^{a}$. But then again, as $\hat{\mathcal{O}}^{+}_{X}(\tilde{Y})\subset \hat{\mathcal{O}}_{X}(\tilde{Y})=\textbf{C}$, we see that $\hat{\mathcal{O}}^{+}_{X}(\tilde{Y})=\mathcal{O}_{\textbf{C}}$.
\end{proof}
\begin{rmk}
If the $Y_{i}$ are connected, one can also directly show that $\tilde{Y}$ is connected as well. For this note that $\tilde{Y}$ is quasi-compact by \cite[3.12 (v)]{Sch1}. Now suppose that $\vert \tilde{Y} \vert=V_{1}\cup V_{2}$ for some open and closed $V_{1}, V_{2}$. Then $V_{1}$ and $V_{2}$ are quasi-compact as closed subsets of a quasi-compact space. But then as quasi-compact opens they are given as pullbacks of open sets $V_{1}^{N}$ and $V_{2}^{N}$ in some $Y_{N}$. Also $V_{1}\to V_{1}^{N}$ (and $V_{2}\to V_{2}^{N}$) is surjective, as it can be written as an inverse limit of surjective maps with finite fibers. Moreover $V_{1}^{N}$ and $V_{2}^{N}$ cover $Y_{N}$ and the intersection $V_{1}^{N}\cap V_{2}^{N}$ is non-empty as $Y_{N}$ is connected. But then, as $V_{1}\to V_{1}^{N}$ and $V_{2}\to V_{2}^{N}$ are surjective, the intersection of $V_{1}$ with $V_{2}$ is non-empty as well.
\end{rmk}
Assume now that $\mathcal{E}^{+}$ is as above with $p$-adic completion $\hat{\mathcal{E}}^{+}$ trivialized on some connected profinite \'etale cover $f:\tilde{Y}=\varprojlim_{i}Y_{i}\to X$, so we have $\Gamma(\tilde{Y}, \hat{\mathcal{E}}^{+})\cong (\mathcal{O}_{\textbf{C}})^{r}$.\\
We will now adjust the exposition in \cite[\S 4]{DW3} to our setting to define an action of $\pi_{1}^{\acute{e}t}(X, x)$ on the fiber $\hat{\mathcal{E}}_{x}^{+}=\Gamma(x^{*}\hat{\mathcal{E}}^{+})$:\\
Pick a point $y:Spa(\textbf{C}, \mathcal{O}_{\textbf{C}})\to \tilde{Y}$ lying over $x$. As $\hat{\mathcal{E}}^{+}$ is trivial on $\tilde{Y}$ we have an isomorphism $y^{*}:\Gamma(\tilde{Y}, \hat{\mathcal{E}}^{+})\xrightarrow{\cong} \Gamma(x^{*}\hat{\mathcal{E}}^{+})$ by pullback (and using the natural identification $y^{*}f^{*}\cong (f\circ y)^{*}=x^{*}$). For any $g\in \pi_{1}(X)$ we get another point $gy$ lying above $x$. We can then define an automorphism on $\hat{\mathcal{E}}^{+}_{x}$ by
\begin{center}
$
\hat{\mathcal{E}}^{+}_{x}\xrightarrow{(y^{*})^{-1}}\Gamma(\tilde{Y}, \hat{\mathcal{E}}^{+})\xrightarrow{(gy)^{*}}\hat{\mathcal{E}}^{+}_{x}
$.
\end{center}
By this we get a map
\begin{center}
$\rho^{(\tilde{Y}, y)}_{\mathcal{E}^{+}}:\pi_{1}(X)\to GL_{r}(\hat{\mathcal{E}}^{+}_{x})$.
\end{center}
\begin{rmk}
What we mean by pullback along $y$ is the following: For any point $x\in X$, valued in $Spa(\textbf{C}, \mathcal{O}_{\textbf{C}})$, a point $y$ of $\tilde{Y}$ over $x$ is given by a compatible system of points $y_{i}:Spa(\textbf{C})\to Y_{i}$ over $x$. By this one gets compatible morphisms of (ringed) sites 
\begin{center}
$\xymatrix{
& (Y_{i})_{pro\acute{e}t}\cong X_{pro\acute{e}t}/Y_{i} \ar[d] \\
Spa(\textbf{C})_{pro\acute{e}t}\ar[ur]^{(y_{i})_{pro\acute{e}t}} \ar[r]^{x_{pro\acute{e}t}} & X_{pro\acute{e}t}.}$
\end{center}
This then gives a morphism $y_{pro\acute{e}t}:Spa(\textbf{C})_{pro\acute{e}t}\to 2-\varprojlim X_{pro\acute{e}t}/Y_{i}\cong X_{pro\acute{e}t}/\tilde{Y}$, lying over $x_{pro\acute{e}t}$. We write $y^{*}$ for the pullback on global sections along $y_{pro\acute{e}t}$. Alternatively, one may carry out the construction within the category of diamonds.
\end{rmk}
We need to show that this defines a continuous representation and is independent of the choices.
\begin{lem}\label{lem23}
Let $\tilde{Y}$, $y$ be as above. Let $\phi:\tilde{Z}\to \tilde{Y}$ be a morphism of connected objects in $X_{prof\acute{e}t}$, and let $z$ be a point lying above $y$. Then $\rho^{(\tilde{Y}, y)}_{\mathcal{E}^{+}}=\rho^{(\tilde{Z}, z)}_{\mathcal{E}^{+}}$.
\end{lem}
\begin{proof}
For any $g\in \pi_{1}(X)$, $gz$ lies above $gy$. Then there is a commutative diagram
\begin{center}
$\xymatrix{
\hat{\mathcal{E}}_{x}^{+}\ar[r]^-{(y^{*})^{-1} }\ar[d]^{=} & \Gamma(\tilde{Y}, \hat{\mathcal{E}}^{+})\ar[r]^-{(gy)^{*}}\ar[d]^{\phi^{*}} & \hat{\mathcal{E}}^{+}_{x}\ar[d]^{=}\\
\hat{\mathcal{E}}_{x}^{+}\ar[r]^-{(z^{*})^{-1}} & \Gamma(\tilde{Z}, \hat{\mathcal{E}}^{+})\ar[r]^-{(gz)^{*}} & \hat{\mathcal{E}}^{+}_{x}
}$
\end{center}
which gives the claim.
\end{proof}
\begin{prop}\label{prop31}
The map $\rho^{(\tilde{Y}, y)}_{\mathcal{E}^{+}}$ mod $p^{n}$ has finite image for all $n\geq 0$.
\end{prop}
\begin{proof}
Let $\tilde{Y}=\varprojlim_{i}Y_{i}$ be a presentation where each $Y_{i}\to X$ is connected finite \'etale. Then $y$ corresponds to a compatible system $y_{i}$ of points of $Y_{i}$.\\
By Lemma \ref{lem31} $\hat{\mathcal{E}}^{+}/p^{n}$ is trivialized on some $Y_{i}\to X$. On the almost level, we then get 
\begin{center}
$\Gamma(\tilde{Y}, \mathcal{E}^{+}/p^{n})^{a}\cong \Gamma(Y_{i}, \mathcal{E}^{+}/p^{n})^{a}\cong (\mathcal{O}^{a}_{\textbf{C}}/p^{n})^{r}$,
\end{center}
where $r$ denotes the rank of $\mathcal{E}^{+}$. The reason here is that $\mathcal{O}_{Y_{j}}^{+a}/p^{n}(Y_{j})\to \mathcal{O}_{Y_{j'}}^{+a}/p^{n}(Y_{j'})$ is an isomorphism for any transition map $Y_{j'}\to Y_{j}$.\\
We thus get an action $\alpha^{(Y_{i}, y_{i})}_{n}$ of $\pi_{1}(X)$ on the $\mathcal{O}_{\textbf{C}}/p^{n}$-module of almost elements $(\hat{\mathcal{E}}^{+}/p^{n})_{*}$ via 
\begin{center}
$
(\hat{\mathcal{E}}^{+}_{x}/p^{n})_{*}\xrightarrow{(y_{i}^{*})^{-1}}\Gamma(Y_{i}, \hat{\mathcal{E}}^{+}/p^{n})_{*}\xrightarrow{(gy_{i})^{*}}((\hat{\mathcal{E}}^{+}_{x}/p^{n})_{*}
$.
\end{center}
But now the natural map $\Gamma(\tilde{Y}, \hat{\mathcal{E}}^{+})/p^{n}\to \Gamma(\tilde{Y}, \mathcal{E}^{+}/p^{n})_{*}=\Gamma(Y_{i}, \mathcal{E}^{+}/p^{n})_{*}$ is injective (after fixing a basis it is just given by the embedding $(\mathcal{O}_{\textbf{C}}/p^{n})^{r}\xhookrightarrow{} (\mathcal{O}_{\textbf{C}}/p^{n})^{r}_{*}$).\\
This realizes $\rho^{(\tilde{Y}, y)}_{\mathcal{E}^{+}} \otimes \mathcal{O}_{\textbf{C}}/p^{n}$ as a subrepresentation of $\alpha^{(Y_{i}, y_{i})}_{n}$. But now there are of course only finitely many points of $Y_{i}$ lying over $x$, so $\alpha^{(Y_{i}, y_{i})}_{n}$ has finite image, hence so has $\rho^{(\tilde{Y}, y)}_{\mathcal{E}^{+}} \otimes \mathcal{O}_{\textbf{C}}/p^{n}$.
\end{proof}
\begin{lem}
The map $\rho^{(\tilde{Y}, y)}_{\mathcal{E}^{+}}$ does not depend on $(\tilde{Y}, y)$.
\end{lem}
\begin{proof}
We only need to show the independence of the point $y$, the rest then follows from Lemma \ref{lem23}.\\
Moreover it is enough to show that $\rho^{(\tilde{Y}, y)}_{\mathcal{E}^{+}}\otimes \mathcal{O}_{\textbf{C}}/p^{n}$ is independent of  $y$ for any $n\geq 1$. Let $\tilde{Y}=\varprojlim Y_{i}$ be a presentation as before, and $y$ correspond to a compatible system of points $y_{i}$ of $Y_{i}$. Assume that $\hat{\mathcal{E}}^{+}/p^{n}$ becomes trivial on $Y_{i}$. Then by the proof of the previous proposition, $\rho^{(\tilde{Y}, y)}_{\mathcal{E}^{+}}\otimes \mathcal{O}_{\textbf{C}}/p^{n}$ is a subrepresentation of $\alpha_{n}^{(Y_{i}, y_{i})}$, so it is enough to show that $\alpha_{n}^{(Y_{i}, y_{i})}$ is independent of the point $y_{i}$ above $x$. Now consider the Galois closure $Y_{i}'\to X$ of $Y_{i}$. Again, as any point $y_{i}$ has a point $y_{i}'$ of $Y_{i}'$ lying above it, by Lemma \ref{lem23} it is enough to show that $\alpha_{n}^{(Y_{i}', y_{i}')}$ is independent of the point $y'$ above $x$. But now the Galois group $G=Aut_{X}(Y'_{i})^{opp}$ of $Y'_{i}\to X$ acts simply transitively on the points lying above $x$, and one can immediately check as in the proof of \cite[Lemma 4.5]{DW3}  that the representation is independent of the point $y_{i}'$. Remark for this that $\alpha_{n}^{(Y_{i}', y_{i}')}$ is given by transporting the $G$-action given by 
\begin{center}
$
\Gamma(f^{*}\hat{\mathcal{E}}^{+}/p^{n})_{*}\xrightarrow{f_{g}^{*}}\Gamma(f_{g}^{*}f^{*}\hat{\mathcal{E}}^{+}/p^{n})_{*}\xrightarrow{can}\Gamma(f^{*}\hat{\mathcal{E}}^{+}/p^{n})_{*}
$
\end{center}
to $(\hat{\mathcal{E}}^{+}_{x}/p^{n})_{*}$ via $(y_{i}')^{*}$ (see also \cite[Proposition 23]{DW1}). Here $f$ denotes the map $Y'_{i}\to X$, $f_{g}$ is the automorphism associated to $g$ and $can$ is induced by the natural identification\\ $f_{g}^{*}f^{*}\cong (f\circ f_{g})^{*}=f^{*}$.
\end{proof}
We thus see that we get a well defined continuous representation $\rho_{\mathcal{E}^{+}}$ associated to $\mathcal{E}^{+}$. Note that one does indeed get a representation, since if $g, h\in \pi_{1}(X)$ and $y$ is any point above $x$ one can write $\rho_{\mathcal{E}^{+}}(g)=(gy)^{*} \circ (y^{*})^{-1}$ and $\rho_{\mathcal{E}^{+}}(h)=(h(gy))^{*}\circ ((gy)^{*})^{-1}$ using that the construction is independent of the chosen point above $x$. From this one then gets 
\begin{center}
$\rho_{\mathcal{E}^{+}}(gh)=\rho_{\mathcal{E}^{+}}(g)\rho_{\mathcal{E}^{+}}(h)$.
\end{center}
We denote by $\mathcal{B}^{p\acute{e}t}(\mathcal{O}^{+}_{X})$ the category of locally free $\mathcal{O}^{+}_{X}$-modules, whose $p$-adic completion is trivialized on a profinite \'etale cover of $X$. One checks:
\begin{lem}
The category $\mathcal{B}^{p\acute{e}t}(\mathcal{O}^{+}_{X})$ is closed under taking tensor products, duals, internal homs, exterior products and extensions.
\end{lem}
\begin{proof}
Assume that we have an extension
\begin{center}
$0\to \mathcal{E}'\to \mathcal{E}\to \mathcal{E}''\to 0$
\end{center}
of $\mathcal{O}_{X}^{+}$-modules, where $\mathcal{E}', \mathcal{E}''\in \mathcal{B}^{p\acute{e}t}(\mathcal{O}^{+}_{X})$. Then by Lemma \ref{lem31} $\mathcal{E}'/p$ and $\mathcal{E}''/p$ become trivial after pulling back to a finite \'etale cover $Y\to X$. Then $\mathcal{E}/p\vert_{Y}$ corresponds to a cohomology class in $H^{1}(\mathcal{O}^{+}_{Y}/p)$. This class again becomes trivial on a further finite \'etale covering (see the proof of Theorem \ref{tf} below). This means that $\mathcal{E}/p$ becomes trivial on a finite \'etale cover. But then $\mathcal{E}\in \mathcal{B}^{p\acute{e}t}(\mathcal{O}^{+}_{X})$ by Theorem \ref{tf}.\\
Closedness under the other operations is left to the reader.
\end{proof}
For any $\mathcal{E}^{+}\in \mathcal{B}^{p\acute{e}t}(\mathcal{O}^{+}_{X})$ we can always find a trivializing cover which is an inverse limit along connected finite \'etale maps:
\begin{lem}
Let $\mathcal{E}^{+}\in \mathcal{B}^{p\acute{e}t}(\mathcal{O}^{+}_{X})$. Then there is a profinite \'etale cover with presentation $\tilde{C}=\varprojlim C_{i}$ such that each $C_{i}$ is connected and $\hat{\mathcal{E}}^{+}$ is trivial on $\tilde{C}$.
\end{lem}
\begin{proof}
First we can assume that $\hat{\mathcal{E}}^{+}$ becomes trivial on $\tilde{Y}$ which has a presentation $\varprojlim_{I} Y_{i}$ along a countable index set $I$(use for this that $\mathcal{E}^{+}/p^{n}$ is trivial on some finite \'etale cover). Then pick a connected component $C_{1}\to Y_{1}$. Pull back $C_{1}$ along $Y_{2}\to Y_{1}$, to get a finite \'etale map $\overline{C}_{2}\to C_{1}$ and again choose a connected component $C_{2}\to \overline{C}_{2}\to C_{1}$. Then continuing like this we get a connected profinite \'etale cover $\varprojlim C_{i}$ with a map $\varprojlim C_{i}\to \tilde{Y}$, so that $\hat{\mathcal{E}}^{+}$ is trivial on $\tilde{C}$.
\end{proof}
\begin{thm}\label{thm21}
The association $\mathcal{E}^{+}\mapsto \rho_{\mathcal{E}^{+}}$ defines an exact functor 
\begin{center}
$\rho_{\mathcal{O}}:\mathcal{B}^{p\acute{e}t}(\mathcal{O}^{+}_{X})\to Rep_{\pi_{1}}(\mathcal{O}_{\textbf{C}})$.
\end{center}
Moreover $\rho_{\mathcal{O}}$ is compatible with tensor products, duals, inner homs and exterior products, and for every morphism $f:X'\to X$ of connected proper rigid analytic spaces over $Spa(\textbf{C}, \mathcal{O}_{\textbf{C}})$ we have $\rho_{f^{*}\mathcal{E}^{+}}=f^{*}\rho_{\mathcal{E}^{+}}$, where $f^{*}\rho_{\mathcal{E}^{+}}$, denotes the composition 
\begin{center}
$\pi_{1}(X', x')\xrightarrow{f_{*}} \pi_{1}(X, f(x'))\xrightarrow{\rho_{\mathcal{E}^{+}}} GL_{r}(\mathcal{O}_{\textbf{C}})$
\end{center}
and $x':Spa(\textbf{C}, \mathcal{O}_{\textbf{C}})\to X'$ is a point of $X'$.
\\[0.2in]
Define by $\mathcal{B}^{p\acute{e}t}(\mathcal{O}_{X}):=\mathcal{B}^{p\acute{e}t}(\mathcal{O}^{+}_{X})\otimes \Q$ the category of finite locally free $\mathcal{O}_{X}$-modules $\mathcal{E}$ for which there exists a locally free $\mathcal{O}_{X}^{+}$-module $\mathcal{E}^{+}$ with $\mathcal{E}^{+}[\frac{1}{p}]=\mathcal{E}$, and $\mathcal{E}^{+}\in \mathcal{B}^{p\acute{e}t}(\mathcal{O}_{X}^{+})$. \\
Passing to isogeny classes induces a functor 
\begin{center}
$
\rho:\mathcal{B}^{p\acute{e}t}(\mathcal{O}_{X})\to Rep_{\pi_{1}(X)}(\textbf{C})
$,
\end{center}
which is compatible with tensor products, duals, inner homs and exterior products.
\end{thm}
\begin{proof}
The functoriality of the construction is clear and the compatibility with the operations follows immediately using the previous lemma.\\
For compatibility with pullbacks, note that if $\hat{\mathcal{E}}^{+}$ is trivial on a connected profinite \'etale cover $\tilde{Y}\to X$, then $\widehat{(f^{*}\mathcal{E}^{+})}$ becomes trivial on the pullback $\tilde{Y}_{X'}\in X_{pro\acute{e}t}'$, which we first assume to be connected. There is a commutative diagram of ringed sites
\begin{center}
$\xymatrix{
X'_{pro\acute{e}t}/\tilde{Y}_{X'} \ar[r]^{pr_{\tilde{Y}}}\ar[d] & X_{pro\acute{e}t}/\tilde{Y}\ar[d] \\
X'_{pro\acute{e}t} \ar[r]^{f_{pro\acute{e}t}} & X_{pro\acute{e}t}
}$
\end{center}
By pullback one has an isomorphism
\begin{center}
$\Gamma(\tilde{Y}, \hat{\mathcal{E}}^{+})\xrightarrow{\cong}\Gamma(\tilde{Y}_{X'}, \widehat{f^{*}\mathcal{E}^{+}})$.
\end{center}
And for any $g\in \pi_{1}(X', x')$ and point $y'$ of $\tilde{Y}_{X'}$ above $x'$, one has $pr_{\tilde{Y}}(gy')=f'_{*}(g)y'$ and there is a commutative diagram
\begin{center}
$\xymatrix@C+4pc{
\hat{\mathcal{E}}_{f(x')}^{+}\ar[r]^-{(pr_{\tilde{Y}}(y')^{*})^{-1} }\ar[d]^{f^{*}} & \Gamma(\tilde{Y}, \hat{\mathcal{E}}^{+})\ar[r]^-{pr_{\tilde{Y}}(gy')^{*}}\ar[d]^{pr_{\tilde{Y}}^{*}} & \hat{\mathcal{E}}^{+}_{f(x')}\ar[d]^{f^{*}}\\
(\widehat{f^{*}\mathcal{E}^{+}})_{x'}\ar[r]^-{(y'^{*})^{-1}} & \Gamma(\tilde{Y}_{X'}, \widehat{f^{*}\mathcal{E}^{+}})\ar[r]^-{(gy')^{*}} & (\widehat{f^{*}\mathcal{E}^{+}})_{x'}.
}$
\end{center}
Now if $\tilde{Y}_{X'}$ is not connected we can again construct a connected cover $\tilde{C}$ as above such that $\widehat{f^{*}\mathcal{E}^{+}}$ is trivial on $\tilde{C}$ and there is a morphism $\tilde{C}\to \tilde{Y}_{X'}$ of objects in $X'_{pro\acute{e}t}$ which gives a commutative diagram
\begin{center}
$\xymatrix{
X'_{pro\acute{e}t}/\tilde{C} \ar[r]^{pr}\ar[d] & X_{pro\acute{e}t}/\tilde{Y}\ar[d] \\
X'_{pro\acute{e}t} \ar[r]^{f_{pro\acute{e}t}} & X_{pro\acute{e}t}
}$
\end{center}
and one can now go on as before.
\end{proof}
\subsection{Weil-Tate local systems}
For this section we assume that $(\textbf{C}, \mathcal{O}_{\textbf{C}})=(\mathbb{C}_{p}, \mathcal{O}_{\C_p})$. Let again $X$ denote a proper connected rigid analytic variety over $Spa(\mathbb{C}_{p}, \mathcal{O}_{\mathbb{C}_p})$. Recall that we have a functor $\mathbb{L}\mapsto \mathbb{L}\otimes \hat{\mathcal{O}}^{+}_{X}$ from $LS_{\mathcal{O}_{\C_p}}(X)$ to $LF(\hat{\mathcal{O}}_{X}^{+})$.
\begin{rmk}
Recall from Proposition \ref{repeq} that there is an equivalence of categories 
\begin{center}
$LS_{\mathcal{O}_{\C_p}}(X)\leftrightarrow Rep_{\pi_{1}(X, x)}(\mathcal{O}_{\C_p})$.
\end{center}
This equivalence can be realized in the following way: For any $\mathbb{L}\in LS_{\mathcal{O}_{\C_p}}(X)$ we can find a profinite \'etale cover $\tilde{Y}=\varprojlim Y_{i}$, where each $Y_{i}$ is connected, and such that $\mathbb{L}$ is trivial on $\tilde{Y}$. Then as above we may define the associated representation as 
\begin{center}
$\mathbb{L}_{x}\xrightarrow{(y^{*})^{-1}} \Gamma(\tilde{Y}, \mathbb{L})\xrightarrow{(gy)^{*}} \mathbb{L}_{x}$.
\end{center}
where $\mathbb{L}_{x}=\Gamma(x_{pro\acute{e}t}^{*}\mathbb{L})=\varprojlim (\tilde{\mathbb{L}}_{n})_{x_{\acute{e}t}}$, where $\tilde{\mathbb{L}}_{n}$ is the \'etale local system with $\nu^{*}\tilde{\mathbb{L}}_{n}=\mathbb{L}/p^{n}$, and $y$ is some point above $x$. This is independent of the cover and chosen point as above. Moreover, modulo $p^{n}$ we get the representation
\begin{center}
$(\mathbb{L}/p^{n})_{x}\xrightarrow{(y_{i}^{*})^{-1}} \Gamma(Y_{i}, \mathbb{L}/p^{n})\xrightarrow{(gy_{i})^{*}} (\mathbb{L}/p^{n})_{x}$,
\end{center} 
where now we can choose $Y_{i}\to X$ to be Galois with Galois group $G$. Then this action is the left action of $G$ on $\Gamma(Y_{i}, \mathbb{L}/p^{n})$ transported to $(\mathbb{L}/p^{n})_{x}$, via $y_{i}^{*}$.\\
Conversely, if $V$ is a finite free $\mathcal{O}_{\C_p}$-module with a continuous $\pi_{1}(X)$-action, we define an inverse system $\{\mathbb{L}_{n} \}$ as follows: Let $\pi_{1}(X)$ act on $V/p^{n}$ through the finite quotient $G$. Let $Y_{n}\to X$ be the finite \'etale Galois cover with group $G$. Then the action of $G$ on $V/p^{n}$ defines a Galois descent datum on the constant sheaf $\underline{V/p^{n}}$ on $Y_{n}$, by letting $\underline{V/p^{n}}\to f_{g}^{*}\underline{V/p^{n}}$ be the map defined by $g^{-1}$ for any $g\in G$, where $f_{g}\in Aut_{X}(Y_{i})=G^{opp}$ denotes the automorphism of $Y_{n}$ associated to $g$. This gives rise to a local system $\mathbb{L}_{n}$ on $X$.
\end{rmk}
\begin{lem}\label{lem24}
Let $\mathcal{E}^{+}\in \mathcal{B}^{p\acute{e}t}(\mathcal{O}_{X}^{+})$ and let $\mathbb{L}$ be the $\hat{\mathcal{O}}_{\mathbb{C}_p}$-local system corresponding to $\rho_{\mathcal{E}^{+}}$. Then $\hat{\mathcal{E}}^{+}\cong \hat{\mathcal{O}}^{+}_{X}\otimes_{\hat{\mathcal{O}}_{\mathbb{C}_p}}\mathbb{L}$.
\end{lem}
\begin{proof}
We can compare the gluing data: Assume that $\hat{\mathcal{E}}^{+}$ and $\mathbb{L}$ are trivialized on $\tilde{Y}\to X$,  a connected profinite \'etale cover. By the discussion above we then have a $\pi_{1}(X,x)$-equivariant isomorphism $\underline{\Gamma(\tilde{Y}, \hat{\mathcal{E}}^{+})}\xrightarrow{\cong}\underline{\hat{\mathcal{E}}_{x}^{+}}=\mathbb{L}\vert_{\tilde{Y}}$, where $\underline{\Gamma(\tilde{Y}, \hat{\mathcal{E}}^{+})}:=\varprojlim\underline{\Gamma(\tilde{Y}, \hat{\mathcal{E}}^{+})/p^{n}}$, where $\underline{\Gamma(\tilde{Y}, \hat{\mathcal{E}}^{+})/p^{n}}$ denotes the constant sheaf on $X_{pro\acute{e}t}/\tilde{Y}$ associated to $\Gamma(\tilde{Y}, \hat{\mathcal{E}}^{+})/p^{n}$ and similarly $\underline{\hat{\mathcal{E}}_{x}^{+}}:=\varprojlim \underline{\hat{\mathcal{E}}_{x}^{+}}/p^{n}$. From this one also gets a $\pi_{1}(X)$-equivariant isomorphism 
\begin{center}
$f:\hat{\mathcal{E}}^{+}\vert_{\tilde{Y}}=\underline{\Gamma(\tilde{Y}, \hat{\mathcal{E}}^{+})}\otimes \hat{\mathcal{O}}^{+}_{\tilde{Y}}\xrightarrow{\cong} (\mathbb{L}\otimes \hat{\mathcal{O}}^{+}_{X})\vert_{\tilde{Y}}$. 
\end{center}
Then modulo $p^{n}$ this equivariant isomorphism descends to some $Y_{i}$, which we can assume to be a finite \'etale Galois cover (if $Y_{i}$ is not Galois we can pull back to the Galois closure $Y_{i}'$, and then extend $Y_{i}'$ to a connected profinite \'etale cover $\tilde{Y}' \to \tilde{Y}$ and pull the whole discussion back to $\tilde{Y}'$). We then get an equivariant isomorphism
\begin{center}
$\hat{\mathcal{E}}^{+}/p^{n}\vert_{Y_{i}}\cong (\hat{\mathcal{O}}^{+}\otimes \mathbb{L})/p^{n}\vert_{Y_{i}}$.
\end{center}
But here the $\pi_{1}(X)$-action is through the quotient $G$, where $G=Aut_{X}(Y_{i})^{opp}$ denotes the Galois group of $Y_{i}\to X$. But then the isomorphism desends to an isomorphism 
\begin{center}
$\hat{\mathcal{E}}^{+}/p^{n}\cong (\hat{\mathcal{O}}^{+}\otimes \mathbb{L})/p^{n}$.
\end{center}
To be a bit more precise we claim that, if $Y_{i}$ is such that $\mathcal{E}^{+}/p^{n}\vert_{Y_{i}}$ and $\mathbb{L}/p^{n}\vert_{Y_{i}}$ are trivial, there is a canonical isomorphism 
\begin{center}
$\mathcal{E}^{+}/p^{n}\vert_{Y_{i}}\cong \Gamma(\tilde{Y}, \mathcal{E}^{+})/p^{n}\otimes \mathcal{O}_{Y_{i}}^{+}/p^{n}$
\end{center}
which takes the canonical gluing datum on the left to the gluing datum obtained from the action on $\Gamma(\tilde{Y}, \mathcal{E}^{+})/p^{n}$, and such that
\begin{center}
$\mathcal{E}^{+}/p^{n}\vert_{Y_{i}}\cong \Gamma(\tilde{Y}, \mathcal{E}^{+})/p^{n}\otimes \mathcal{O}_{Y_{i}}^{+}/p^{n}\to (\hat{\mathcal{O}}^{+}\otimes \mathbb{L})/p^{n}\vert_{Y_{i}}$
\end{center}
descends $f$ mod $p^{n}$ to $Y_{i}$. For this, note that there is a commutative diagram 
\begin{center}
$\xymatrix{
 & \Gamma(Y_{i}, \mathcal{E}^{+}/p^{n}) \ar[r]\ar[d]& \Gamma(Y_{i}, \mathcal{E}^{+}/p^{n})_{*} \ar[d]^{\cong}\\
 \Gamma(\tilde{Y}, \hat{\mathcal{E}}^{+})/p^{n}\ar[r] & \Gamma(\tilde{Y}, \mathcal{E}^{+}/p^{n})\ar[r] & \Gamma(\tilde{Y}, \mathcal{E}^{+}/p^{n})_{*}
 }
 $
 \end{center}
where the vertical maps are just given by pulling back sections. This proves that the image of $\Gamma(\tilde{Y}, \hat{\mathcal{E}}^{+})/p^{n}$ in $\Gamma(Y_{i}, \mathcal{E}^{+}/p^{n})_{*}$ is contained in the image of $\Gamma(Y_{i}, \mathcal{E}^{+}/p^{n})$. For this note that $\Gamma(Y_{i}, \mathcal{O}_{Y_{i}}^{+}/p^{n})$ always contains a canonical copy of $\mathcal{O}_{\C_p}/p^{n}$ coming from the base. Fixing a trivialization of $\hat{\mathcal{E}}^{+}\vert_{\tilde{Y}}$ one sees that then the canonical copy of $(\mathcal{O}_{\C_p}/p^{n})^{r}$ ($r$ is the rank of $\mathcal{E}^{+}$) in $\Gamma(Y_{i}, \mathcal{E}^{+}/p^{n})$ is identified with $\Gamma(\tilde{Y}, \hat{\mathcal{E}}^{+})/p^{n}$ through the diagram above.\\
Moreover, as $\mathcal{E}^{+}/p^{n}$ is trivial on $Y_{i}$, the canonical map $\Gamma(Y_{i}, \mathcal{E}^{+}/p^{n})\otimes \mathcal{O}^{+}_{Y_{i}}/p^{n}\to \mathcal{E}^{+}/p^{n}\vert_{Y_{i}}$ is surjective and its kernel is almost zero (as it becomes an isomorphism after passing to almost modules). This produces a unique (injective) map $g$ fitting into the commutative diagram
\begin{center}
$\xymatrix{
\Gamma(Y_{i}, \mathcal{E}^{+}/p^{n})\otimes \mathcal{O}^{+}_{Y_{i}}/p^{n}\ar[r]\ar[dr] & \mathcal{E}^{+}/p^{n}\vert_{Y_{i}}\ar[d]^{g} \\
& \Gamma(Y_{i}, \mathcal{E}^{+}/p^{n})_{*}\otimes \mathcal{O}^{+}_{Y_{i}}/p^{n}.
}$
\end{center}
From the first diagram one sees that one then gets an injective map 
\begin{center}
$\phi: \Gamma(\tilde{Y}, \hat{\mathcal{E}}^{+})/p^{n} \otimes \mathcal{O}^{+}_{Y_{i}}/p^{n}\to \mathcal{E}^{+}/p^{n}\vert_{Y_{i}}$
\end{center}
of submodules of $\Gamma(\tilde{Y}, \hat{\mathcal{E}}^{+}/p^{n})_{*}\otimes \mathcal{O}^{+}_{Y_{i}}/p^{n}$. But now, as $\phi$ becomes an isomorphism after passing to almost modules, it must also already be surjective itself (as the image is finitely generated the cokernel could not be almost zero otherwise). Thus we have a canonical isomorphism
\begin{center}
$\phi^{-1}:\mathcal{E}^{+}/p^{n}\vert_{Y_{i}} \xrightarrow{\cong} \Gamma(\tilde{Y}, \hat{\mathcal{E}}^{+})/p^{n} \otimes \mathcal{O}^{+}_{Y_{i}}/p^{n}$
\end{center}
and one can now check that the canonical gluing datum on $\mathcal{E}^{+}/p^{n}\vert_{Y_{i}}$ is transported to the one coming from the action on $\Gamma(\tilde{Y}, \hat{\mathcal{E}}^{+})/p^{n}$ (recall that the $G$-action on $\Gamma(\tilde{Y}, \hat{\mathcal{E}}^{+})/p^{n}$ comes from the action on $\Gamma(Y_{i}, \hat{\mathcal{E}}^{+}/p^{n})_{*}$ which is induced by the left action of $G^{opp}$ on $Y_{i}$).\\
So all in all, $f$ mod $p^{n}$ descends to $X$ for all $n$, hence so does $f$ itself.
\end{proof}
Denote by $LF^{p\acute{e}t}(\hat{\mathcal{O}}_{X}^{+})$ the category of profinite \'etale trivializable modules over $\hat{\mathcal{O}}_{X}^{+}$ (we remark that by Theorem \ref{tf} this category is equivalent to the category of $\hat{\mathcal{O}}_{X}^{+}$-modules $\hat{\mathcal{E}}^{+}$ for which $\hat{\mathcal{E}}^{+}/p^{n}$ is trivial on a finite \'etale cover for some $n\geq 0$). The discussion from the previous section gives a functor 
\begin{center}
$\hat{\rho}_{\mathcal{O}}:LF^{p\acute{e}t}(\hat{\mathcal{O}}_{X}^{+})\to Rep_{\pi_{1}(X)}(\mathcal{O}_{\C_p})\to LS_{\mathcal{O}_{\C_p}}(X)$ 
\end{center}
(where the second functor is the equivalence from Proposition \ref{repeq}).
\begin{cor}\label{full}
The functor $LS_{\mathcal{O}_{\C_p}}(X)\to LF^{p\acute{e}t}(\hat{\mathcal{O}}_{X}^{+})$, taking $\mathbb{L}$ to $\hat{\mathcal{O}}_{X}^{+}\otimes \mathbb{L}$, is an equivalence of categories with quasi-inverse given by $\hat{\rho}_{\mathcal{O}}$.
\end{cor}
\begin{proof}
First note that any $\mathcal{O}_{\C_p}$-local system $\mathbb{L}$ becomes trivial on a profinite \'etale cover, so that $\mathbb{L}\otimes \hat{\mathcal{O}}^{+}_{X}\in LF^{p\acute{e}t}(\hat{\mathcal{O}}_{X}^{+})$. Using the lemma above one can then check that $\hat{\rho}_{\mathcal{O}}$ is a quasi-inverse.\\
From the remark above we know that if $\tilde{Y}$ is a connected profinite \'etale cover trivializing $\mathbb{L}$, then the representation associating to $g\in \pi_{1}(X, x)$ the automorphism
\begin{center}
$\mathbb{L}_{x}\xrightarrow{(y^{*})^{-1}} \Gamma(\tilde{Y}, \mathbb{L})\xrightarrow{(gy)^{*}} \mathbb{L}_{x}$
\end{center}
is the representation associated to $\mathbb{L}$ in Proposition \ref{repeq}. But then the commutative diagram 
\begin{center}
$
\xymatrix{
\mathbb{L}_{x}\ar[r]^-{(y^{*})^{-1}}\ar[d]^{=} &  \Gamma(\tilde{Y}, \mathbb{L})\ar[r]^{(gy)^{*}}\ar[d]^{\cong} & \mathbb{L}_{x}\ar[d]^{=}\\
(\hat{\mathcal{O}}^{+}_{X}\otimes \mathbb{L})_{x}\ar[r]^{(y^{*})^{-1}} & \Gamma(\tilde{Y}, \hat{\mathcal{O}}^{+}_{X}\otimes \mathbb{L})\ar[r]^-{(gy)^{*}} & (\hat{\mathcal{O}}^{+}_{X}\otimes \mathbb{L})_{x}
}
$
\end{center}
shows that $\hat{\rho}(\mathbb{L}\otimes \hat{\mathcal{O}}_{X}^{+})=\mathbb{L}$.
\end{proof}
Using results by Kedlaya-Liu we can then also show that $\mathcal{\rho}$ is fully faithful.
\begin{thm}\label{ff}
Assume that $X$ is seminormal. Then the functor 
\begin{center}
$\rho:\mathcal{B}^{p\acute{e}t}(\mathcal{O}_{X})\to Rep_{\pi_{1}(X)}(\mathbb{C}_{p})$
\end{center}
is fully faithful.
\end{thm}
\begin{proof}
The functor $\rho$ is given by composing $E\to E\otimes \hat{\mathcal{O}}$ with the functor $\hat{\rho}_{\mathcal{O}}\otimes \Q$. The latter is fully faithul by the previous corollary, while the full faithfulness of the functor $E\to E\otimes \hat{\mathcal{O}}$
follows from \cite[Corollary 8.2.4]{KLiu}, since $X$ is seminormal.
\end{proof}
We see that the vector bundles $E$ in $\mathcal{B}^{p\acute{e}t}(\mathcal{O}_{X})$ are precisely the vector bundles for which there exists an $\hat{\mathcal{O}}_{\C_p}$-local system $\mathbb{L}$ such that $E\otimes \hat{\mathcal{O}}_{X}\cong \hat{\mathcal{O}}_{X}\otimes \mathbb{L}$. We borrow the following terminology from \cite[Definition 10.3]{Xu}.
\begin{defn}\label{wt}
A vector bundle $E$ is called Weil-Tate if there exists an $\hat{\mathcal{O}}_{\C_p}$-local system $\mathbb{L}$ such that $E\otimes \hat{\mathcal{O}}_{X}\cong \hat{\mathcal{O}}_{X}\otimes \mathbb{L}$.\\
Similarly, the $\hat{\mathcal{O}}_{\C_p}$-local systems $\mathbb{L}$ that are associated to vector bundles in this sense are also called Weil-Tate
\end{defn}
\begin{rmk}
Let $X$ be a smooth rigid analytic variety over $Spa(K, \mathcal{O}_{K})$ for $K$ a finite extension of $\Q_{p}$. In \cite{LZ} Liu and Zhu have introduced a functor from the category of arithmetic $\hat{\Z}_{p}$-local systems on $X$ to the category of nilpotent Higgs bundles on $X_{\hat{\overline{K}}}$. Recall that a Higgs-bundle on $X_{\hat{\overline{K}}}$ with values in $\Omega^{1}_{X_{\hat{\overline{K}}}}(-1)$ (here $(-1)$ denotes the Tate twist) is a vector bundle $E$ on $X_{\hat{\overline{K}}}$ together with an endomorphism valued $1$-form, i.e. a section $\phi\in H^{0}(\mathcal{E}nd(E)\otimes  \Omega^{1}(-1))$, satisfying the integrability condition $\phi\wedge \phi=0$. We view the coherent sheaves here as sheaves on the \'etale site.\\
We give a quick recollection of their construction. Let $\mathcal{O}\mathbb{B}_{dR}$ be the de Rham structure sheaf defined in \cite{Schc}. It carries a flat connection $\nabla$ (acting trivially on $\mathbb{B}_{dR}$) and a filtration coming from the usual filtration on $\mathbb{B}_{dR}$. Then define $\mathcal{O}\C=gr^{0}(\mathcal{O}\mathbb{B}_{dR})$.\\
Let $\nu':X_{pro\acute{e}t}/X_{\hat{\overline{K}}}\to (X_{\hat{\overline{K}}})_{\acute{e}t}$ denote the natural projection. Then for any $\hat{\Z}_{p}$-local system $\mathbb{L}$ the associated Higgs bundle is defined to be $\mathcal{H}(\mathbb{L})=\nu'_{*}(\mathcal{O}\C\otimes_{\hat{\Z}_p} \mathbb{L})$, with Higgs field $\theta_{\mathbb{L}}$ coming from the Higgs field on $\mathcal{O}\C$. Note that $\mathcal{O}\C$ carries a Higgs field with values in $\Omega^{1}(-1)$ coming from the associated graded of $\nabla$.\\
One may try to write down this functor for geometric local systems, i.e. we will consider the functor $\mathcal{H}(\mathbb{L})=\nu'_{*}(\mathcal{O}\C\otimes_{\hat{\mathcal{O}}_{\C_p}} \mathbb{L})$ for $\mathcal{O}_{\C_p}$-local systems on $X_{\hat{\bar{K}}}$.\\
We wish to show that our constructions are compatible with this functor. In general, proving that this functor actually gives a Higgs bundle is a complicated endevour, and is what occupies the large part of \S 2 in \cite{LZ} (for arithmetic local systems). In our case however this is immediate, as the local system is already attached to a vector bundle.\\ 
More precisely we have the following result (compare also with \cite[Proposition 11.7]{Xu}):
\begin{prop}
Let $X$ be a smooth proper rigid analytic variety over $Spa(K, \mathcal{O}_{K})$. Let $E$ be a Weil-Tate vector bundle on $X_{\hat{\overline{K}}}$, viewed as a sheaf on the \'etale site. Let $\mathbb{L}$ be the $\hat{\mathcal{O}}_{\C_{p}}$-local system associated to $E$. \\
Then $\nu'_{*}(\mathcal{O}\C\otimes\mathbb{L})\cong E$, with vanishing Higgs field $\theta_{\mathbb{L}}=0$.
\end{prop}
\begin{proof}
We have $\widehat{(\nu'^{*}E)}\cong \hat{\mathcal{O}}_{X}\otimes \mathbb{L}$. We have $gr^{0}\mathbb{B}_{dR}=\hat{\mathcal{O}}_{X_{\hat{\overline{K}}}}$, hence $\hat{\mathcal{O}}_{X_{\hat{\overline{K}}}}\subset \mathcal{O}\C$ (see \cite[Prop. 6.7]{Sch1}). Then $\mathcal{O}\C\otimes \mathbb{L}=\widehat{(\nu'^{*}E)}\otimes_{\hat{\mathcal{O}}_{X_{\hat{\overline{K}}}}}\mathcal{O}\C$. Since $\nu_{*}\mathcal{O}\C=\mathcal{O}_{X_{\acute{e}t}}$, we get $\nu'_{*}(\mathcal{O}\C\otimes\mathbb{L})\cong E$. \\
For the claim $\theta_{\mathbb{L}}=0$, note that the Higgs field on $\mathcal{O}\C$ is trivial on $\hat{\mathcal{O}}_{X_{\hat{\overline{K}}}}$ (as it comes from the connection $\nabla$ on $\mathcal{O}\mathbb{B}_{dR}$ which is trivial on $\mathbb{B}_{dR}$). But then the induced Higgs field on $\mathcal{H}(\mathbb{L})=\nu'_{*}(\widehat{(\nu'^{*}E)}\otimes_{\hat{\mathcal{O}}_{X_{\hat{\overline{K}}}}}\mathcal{O}\C)$ is simply obtained by trivially extending the Higgs field from $\nu'_{*}\mathcal{O}\C=\mathcal{O}_{X_{\acute{e}t}}$. But this is the zero Higgs field.
\end{proof}
\end{rmk}
\subsection{Frobenius-trivial modules}
In this section we will deal with $\mathcal{O}^{+}_{X}$-modules whose mod $p$ reduction is trivialized by some Frobenius pullback. Let $\mathcal{E}$ be an $\mathcal{O}^{+}_{X}/p$-module. Then  
$\mathcal{E}$ is called $F^{m}$-trivial if $\Phi^{m*}\mathcal{E}\cong (\mathcal{O}^{+}_{X}/p)^{r}$, where $\Phi$ denotes the Frobenius on $\mathcal{O}^{+}_{X}/p$. The goal of this section is to show that all such modules lie in $\mathcal{B}^{p\acute{e}t}(\mathcal{O}_{X}^{+})$.
\\
We have the following (Recall that $t\in \mathcal{O}_{\textbf{C}}^{\flat}$, such that $ t^{\sharp}= p $):
\begin{lem}\label{flem}
The Frobenius induces an equivalence of categories
\begin{center}
$
\{F^{m}$-trivial locally free $\mathcal{O}^{+}_{X}/p$-modules $\}$
\\
$\longleftrightarrow $
\\
$\{$locally free $\hat{\mathcal{O}}_{X^{\flat}}^{+}/t^{p^m}$-modules trivial mod $t\}$.
\end{center}
\end{lem}
\begin{proof}
Denote the Frobenius on $\hat{\mathcal{O}}^{+}_{X^{\flat}}$ by $\tilde{\Phi}$. There is a commutative diagram
\begin{center}
$\xymatrix{
\hat{\mathcal{O}}^{+}_{X^{\flat}} \ar[d]^{\theta}\ar[r]^{\tilde{\Phi}}  & \hat{\mathcal{O}}^{+}_{X^{\flat}}\ar[d]^{\theta} \\
\hat{\mathcal{O}}^{+}_{X^{\flat}}/t \ar[r]^{\Phi} & \hat{\mathcal{O}}^{+}_{X^{\flat}}/t.
}$
\end{center}
As $\theta$ is surjective, we see that $\theta^{*}\theta_{*}\mathcal{M}=\mathcal{M}$ for any $\hat{\mathcal{O}}_{X^{\flat}}^{+}/t$-module $\mathcal{M}$. Here $\theta_{*}$ denotes the restriction of scalars. But then using the commutativity of the diagram above we see that $(\tilde{\Phi}^{m})^{*}$ defines the desired functor with quasi-inverse given by $(\tilde{\Phi}^{-m})^{*}$.
\end{proof}
\begin{rmk}
One also sees in the same manner that we also have an equivalence of categories
\begin{center}
$\{$ mod $p$ $F^{m}$-trivial locally free $\hat{\mathcal{O}}^{+}_{X}$-modules $\}$
\\
$\longleftrightarrow $
\\
$\{$ locally free $\mathbb{A}_{inf, X}/\tilde{\Phi}(\xi)$-modules trivial mod $(\xi, p)\}$
\end{center}
where $\mathbb{A}_{inf, X}=W(\hat{\mathcal{O}}^{+}_{X^{\flat}})$ and now $\tilde{\Phi}$ denotes the Frobenius on $\mathbb{A}_{inf, X}$. 
\end{rmk}
\begin{thm}\label{tf}
Let $\mathcal{E}^{+}$ be a locally free $\mathcal{O}^{+}_{X}$-module of rank $r$ such that $\mathcal{E}^{+}/p$ is $F^{m}$-trivial. Then there exists a profinite \'etale cover $\tilde{Y}=\varprojlim_{n} Y_{n}\to X$ such that $\hat{\mathcal{E}}^{+}\vert_{\tilde{Y}}\cong (\hat{\mathcal{O}}_{X}^{+}\vert_{\tilde{Y}})^{r}$.
\end{thm}
\begin{proof}
Assume first that $m=0$, so that $\mathcal{E}^{+}/p$ is trivial. We want to show that $\mathcal{E}^{+}/p^{n}$ becomes trivial after passing to a further finite \'etale cover. We have an exact sequence:
\begin{center}
$
0\to M_{n}(\mathcal{I})\to GL_{r}(\mathcal{O}_{X}^{+}/p^{2})\to GL_{r}(\mathcal{O}_{X}^{+}/p)\to 1
$
\end{center}
where $\mathcal{I}=(p)\mathcal{O}_{X}^{+}/p^{2}$ and the first map is given by $A\mapsto 1+A$. Taking cohomology we get an exact sequence
\begin{center}
$
H^{1}(M_{n}(\mathcal{I}))\to H^{1}(GL_{r}(\mathcal{O}_{X}^{+}/p^{2}))\to H^{1}(GL_{r}(\mathcal{O}_{X}^{+}/p))
$.
\end{center}
Pick a pro-\'etale cover $\{U_{i} \}$ of $X$ on which $\mathcal{E}^{+}$ becomes trivial. Using the exact sequence above plus the fact that $\mathcal{E}^{+}/p$ is trivial we see that $\mathcal{E}^{+}/p^{2}$ is (after possibly refining the cover $\{U_{i} \}$) defined by a cocycle of the form $(id+g_{ij})_{ij}$ on the overlaps $U_{i}\times_{X}U_{j}$, where $(g_{ij})$ defines a class in $H^{1}(M_{n}(\mathcal{I}))$. Since $\mathcal{O}^{+}_{X}$ is $p$-torsion free we have an isomorphism of pro-\'etale sheaves $\mathcal{I}\cong \mathcal{O}^{+}_{X}/p$. Hence by the primitive comparison theorem \ref{cisoo} we see that $H^{1}(M_{n}(\mathcal{I}))=M_{n}(H^{1}(\mathcal{I}))$ is almost isomorphic to $M_{n}(H^{1}(X_{\acute{e}t}, \mathbb{F}_{p})\otimes \mathcal{O}_{\textbf{C}}/p)$. But the classes in the latter cohomology group become zero on suitable finite \'etale covers. Hence we can assume that the class defined by $(g_{ij})$ is almost trivial, which means that $p^{\epsilon}g_{ij}$ becomes a coboundary for any $\epsilon \in log\Gamma$. Write $p^{\epsilon}g_{ij}=p(\gamma_{j}-\gamma_{i})$, where the $\gamma_{i}$ are matrices with entries in $\mathcal{O}^{+}/p^{2}(U_{i})$. Then $p^{1-\epsilon}(\gamma_{j}-\gamma_{i})-g_{ij}$ is divisible by $p^{2-\epsilon}$. Hence $g_{ij}$ is given by a coboundary modulo $p^{2-\epsilon}$, so $\mathcal{E}^{+}/p^{2-\epsilon}$ is trivial.
Inductively we see that $\mathcal{E}^{+}$ can be trivialized on suitable finite \'etale covers modulo $p^{n-\sum_{k=1}^{n}\epsilon_{k}}$ where we can choose the sequence $\epsilon_{k}$ in such a way that $\epsilon_{k}\to 0$ as $k$ goes to infinity, hence giving the claim.
\\
\\
Now assume that $\Phi^{m*}(\mathcal{E}^{+}/p)$ is trivial for some $m>0$. Using Lemma \ref{flem} we see that $\mathcal{F}=\tilde{\Phi}^{m*}(\mathcal{E}^{+}/p)$ is a locally free $\hat{\mathcal{O}}_{X^{\flat}}^{+}/t^{p^m}$-module trivial mod $t$. The obstruction for triviality of $\mathcal{F}/t^{2}$ lies again in $M_{n}(H^{1}(\hat{\mathcal{O}}^{+}_{X^{\flat}}/t))=M_{n}(H^{1}(\mathcal{O}^{+}_{X}/p))$. Now applying the same arguments as in the first part of the proof, we see that $\mathcal{F}$ becomes trivial on a finite \'etale cover $Y\to X$. But then, applying $(\tilde{\Phi}^{-m})^{*}$, we see that $\mathcal{E}^{+}/p$ becomes trivial on $Y$ as well.
\end{proof}
We found out that the idea for the first part of the proof of the last theorem is essentially already contained in \cite[\S 5]{Fa}.\\
For generalities on non-abelian cohomology on sites we refer to \cite{gir} (see in particular \cite[III Proposition 3.3.1]{gir}).
\section{The Deninger-Werner correspondence for rigid analytic varieties}
We wish to use the results from the previous section to give a new approach to the Deninger-Werner correspondence, which works for general (seminormal) proper rigid analytic varieties. 
\subsection{Numerically flat vector bundles}
We first need to generalize the results from \cite[\S 2]{DW3} on numerically flat vector bundles over finite fields to the non-projective case. For any $\mathbb{F}_{p}$-scheme $Y$ we will denote by $F_{Y}$ the absolute Frobenius morphism of $Y$. For a vector bundle $E$ on $Y$ we denote its dual by $E^{\vee}$. Recall that a vector bundle $E$ on a smooth projective curve $C$ is called semistable if for all subbundles $E'\subset E$ we have $\mu(F')\leq \mu(F)$, where $\mu=\frac{deg}{rk}$ denotes the slope.
\begin{prop}\label{propnflat}
Let $k$ be a perfect field. Let $Y$ be a proper, connected $k$-scheme and $E$ a vector bundle on $Y$. Then the following are equivalent:
\begin{itemize}
\item For all morphisms $f:C\to Y$ from a smooth projective curve $C$, we have that $f^{*}E$ is semistable of degree $0$. (This is also called Nori-semistability)
\item The canonical line bundles $\mathcal{O}_{\mathbb{P}(E)}(1)$, resp. $\mathcal{O}_{\mathbb{P}(E^{\vee})}(1)$ on the associated projective bundle $\mathbb{P}(E)$, resp. $\mathbb{P}(E^{\vee})$ are numerically effective.
\end{itemize}
\end{prop}
\begin{defn}
A vector bundle $E$ satisfying the equivalent conditions in Proposition \ref{propnflat} is called numerically flat.
\end{defn}
\begin{rmk}
\begin{itemize}
\item Assume for simplicity that $Y$ is a smooth, projective curve. Over a characteristic $0$ field one can check that numerically flat vector bundles are simply the semistable vector bundles of degree $0$.\\
One of the main complications in the theory of vector bundles over a field of positive characteristic is the fact that a semistable bundle might become unstable after pullback along an inseparable morphism (in contrast to this, semistability is preserved under pullback along any separable map). In characteristic $p$ the numerically flat vector bundles coincide with so called strongly semistable vector bundles of degree $0$, i.e. bundles $E$ for which $F^{n*}E$ is semistable of degree $0$ for all $n\geq 0$.
\item In contrast to the category of degree $0$ semistable vector bundles, the category of numerically flat bundles is still well behaved over a field of positive characteristic. In particular it is a neutral Tannakian category, which has been extensively studied in \cite{La}.
\item A standard example of a semistable vector bundle in positive characteristic which becomes unstable after Frobenius pullback is given by $F_{C*}\mathcal{O}_{C}$ for a smooth projective curve $C$ of genus $\geq 2$. In this case a direct computation of the degree shows that $F_{C}^{*}F_{C*}\mathcal{O}_{C}\rightarrow \mathcal{O}_{C}$ destabilizing.
\end{itemize}
\end{rmk}
We record the following 
\begin{lem}\label{lem41}
Let $Y$ be a proper connected scheme over a perfect field $k$. Then a vector bundle $E$ is numerically flat on $Y$ if and only if $f^{*}E$ is numerically flat for any proper surjective morphism $f:Z\to Y$.\\
Moreover $E$ is numerically flat if and only if $E_{k'}$ is numerically flat for any field extension $k'/k$.
\end{lem}
\begin{proof}
Using the characterization of numerically flat bundles via Nori-semistability we see that we are reduced to show that a vector bundle $E$ on a smooth projective curve $C$ is strongly semistable of degree $0$ if and only if $f^{*}E$ is strongly semistable for some finite map of smooth curves $f:C'\to C$. But this is a standard result in the theory of vector bundles on curves.\\
The second statement follows similarly from the invariance of semistablity with respect to arbitrary field extensions (see \cite[Theorem 1.3.7]{HL}).
\end{proof}
In particular we see that if $E$ is numerically flat, $F_{Y}^{n*}E$ is also numerically flat  for all $n\geq 0$.\\
The main goal of this section is to generalize a structure theorem for numerically flat bundles (\cite[Theorem 2.2]{DW3}) to non-projective proper schemes. As we will later study formal models over $Spf(\mathcal{O}_{\C_p})$ we will actually immediately deal with the situation of a proper scheme over $\mathcal{O}_{\C_p}/p$.
\begin{thm}\label{thm31}
Let $Y$ be a proper connected scheme over $Spec(\mathcal{O}_{\C_p}/p)$ and let $E$ be a vector bundle on $Y$.\\
Then $E\otimes \bar{\F}_{p}$ is numerically flat on $Y\times_{Spec(\mathcal{O}_{\C_p}/p)}Spec(\bar{\F}_{p})$ if and only if there exists a finite \'etale cover $f:Y'\to Y$, and an $e\geq 0$ such that $F_{Y'}^{e*}f^{*}E\cong \mathcal{O}_{Y'}^{r}$.
\end{thm}
The proof for projective schemes over $\bar{\mathbb{F}}_{p}$ in \cite{DW3} relies on Langer's boundedness theorem for semistable sheaves (see \cite{La1}) . Using these results Deninger-Werner show the following
\begin{prop}\cite[Theorem 2.4]{DW3}\label{prop32}
Let $Y$ be a projective connected scheme over $\mathbb{F}_{q}$. Then the set of isomorphism classes of numerically flat vector bundles of fixed rank $r$ on $Y$ is finite.
\end{prop}
From this it follows that for any numerically flat bundle $E$ there exist numbers $r>s\geq 0$, such that $F_{Y}^{r*}E\cong F_{Y}^{s*}E$. One then concludes with the following
\begin{thm}\cite[Satz 1.4]{LS}\cite[Proposition 4.1]{Ka}
Let $Z$ be any $\F_{p}$-scheme and $G$ a vector bundle on $Z$ for which there exists an isomorphism $F_{Z}^{n*}G\cong G$ for some $n>0$. Then there is a finite \'etale cover of $Z$ on which $G$ becomes trivial.
\end{thm}
As the author knows of no way to bound vector bundles on non-projective schemes, we are not able to show finiteness of isomorphism classes as in Proposition \ref{prop32}.\\
The proof of Theorem \ref{thm31} will instead be an application of $v$-descent for perfect schemes as established in \cite{BS}. For this we will briefly recall the necessary ingredients.
\begin{defn}
An $\F_{p}$-scheme $Z$ is called perfect if $F_{Z}$ is an automorphism. The category of perfect $\F_{p}$-schemes will be denoted by $Perf$.
\end{defn}
\begin{rmk}
\begin{itemize}
\item As explained in \cite[\S 3]{BS} there is a functor $Z\mapsto Z_{perf}$ from $\F_{p}$-schemes to the category of perfect schemes, where $Z_{perf}=\varprojlim_{F_{Z}}Z$ denotes the inverse limit along the absolute Frobenius morphism.
\item The topology on $Perf$ studied in loc. cit. is the so called $v$-topology, which is a non-noetherian version of the $h$-topoogy. For its definition we refer to \cite[\S 2]{BS}. The only thing we need here, is that for any proper surjective cover $f:Z'\to Z$ of $\F_{p}$-schemes, $f_{perf}$ is a $v$-cover.
\end{itemize}
\end{rmk}
The main result for us is then:
\begin{thm}\cite[Theorem 4.1]{BS}\label{thmbs}
Let $Vect_{r}(-)$ denote the groupoid of vector bundles of rank $r$. The association $Z\mapsto Vect_{n}(Z)$ is a $v$-stack on $Perf$.
\end{thm}
We are now ready to prove the main result of this section:
\begin{proof}[Proof of Theorem \ref{thm31}]
Let $Y$ be proper connected over $\mathcal{O}_{\C_p}/p$ and $E$ a vector bundle on $Y$. By standard descent results for finitely presented modules, we can assume that $Y$ and $E$ descend to $(Y', E')$ over $\mathcal{O}_{K}/p$ for some finite extension $K/\Q_{p}$. Let $\pi\in \mathcal{O}_{K}$ be a uniformizer and $\mathcal{O}_{K}/\pi=\F_{q}$. Assume first that $Y'$ is projective. Then the following argument is essentially already contained in the proof of Theorem 2.2 in \cite{DW3}:\\
By Proposition \ref{prop32} there are only finitely many numerically flat bundles on $Y'\times Spec(\F_{q})$ up to isomorphism. But $Y'$ is an infinitesimal thickening of $Y'\times Spec(\F_{q})$ (as $\mathcal{O}_{K}/p$ is an Artin ring). But then the lifts of a fixed vector bundle $G$ on $Y'\times Spec(\F_{q})$ to $Y'$ are parametrized by a finite dimensional vector space over $\F_{q}$. This means that there are only finitely many vector bundles of rank $r$ on $Y'$ whose reduction mod $\pi$ is numerically flat. As $F_{Y'}^{n*}E'$ lies in this set for all $n\geq 0$, we find some natural numbers $r>s\geq 0$ such that $F_{Y'}^{r*}E'\cong F_{Y'}^{s*}E'$.\\
Now assume that $Y'$ is proper but not projective. By Chow's lemma we can find a proper surjective cover $f:Z\to Y'$ where $Z$ is projective over $\mathcal{O}_{K}/p$. By Lemma \ref{lem41} $f^{*}E'$ is a numerically flat vector bundle. We have the canonical gluing datum $\phi_{can}:pr_{1}^{*}(f^{*}E')\to pr_{2}^{*}(f^{*}E')$ where $pr_{1}, pr_{2}:Z\times_{Y'}Z\to Z$ denote the canonical projections.
\begin{claim}
The set $M=\{$ descent data $(G, \phi)$ wrt $f$ where $G$ is numerically flat of $rk$ $r$ $\}/Iso$ is finite.
\end{claim}
The claim follows from the fact that $G$ runs through finitely many isomorphism classes and $\phi$ lies in the finite $\F_{q}$-vector space $Hom_{Z\times_{Y'}Z}(pr_{1}^{*}G, pr_{2}^{*}G)$.\\
As Frobenius commutes with all maps, $F_{Z}$ acts on $M$. Hence we get an isomorphism 
\begin{center}
$\Psi:F_{Z}^{r*}(f^{*}E', \phi_{can})\cong F_{Z}^{s*}(f^{*}E', \phi_{can})$
\end{center}
for some natural numbers $r>s\geq0$. Now of course $\Psi$ will in general not descend to an isomorphism between $F_{Y'}^{r*}E'$ and $F_{Y'}^{s*}E'$. But by Theorem \ref{thmbs}, after passing to the perfection, we see that $f_{perf}:Z_{perf}\to Y'_{perf}$ satisfies effective descent for vector bundles. Hence $\pi_{Z}^{*}(\Psi)$ descends to an isomorphism $\pi_{Y'}^{*}(F_{Y'}^{r*}E')\cong \pi_{Y'}^{*}(F_{Y'}^{r*}E')$ where $\pi_{Z}:Z_{perf}\to Z$ and $\pi_{Y'}:Y'_{perf}\to Y'$ denote the canonical projections.\\
But the category of (descent data of) vector bundles on $Z_{perf}$ is the colimit of the categories of (descent data of) vector bundles on copies of $Z$ along Frobenius pullbacks. But then we see that $\Psi$ already becomes effective after a high enough Frobenius pullback. This gives an isomorphism $F_{Y'}^{n*}F_{Y'}^{r*}E'\cong F_{Y'}^{n*}F_{Y'}^{s*}E'$ for some $n>>0$.
\end{proof}
\subsection{The Deninger-Werner correspondence}
In this section we will prove our main result, which is the construction of $p$-adic representations attached to vector bundles with numerically flat reduction on an arbitrary proper (seminormal) rigid analytic variety $X$. Moreover we will later show that our representations coincide with the ones constructed in \cite{DW3} whenever $X$ is the analytification of a smooth algebraic variety over $\overline{\Q}_{p}$.
\\
We will treat the integral and rational case simultaneously. So let $\mathcal{X}$ be a proper connected admissible formal scheme over $Spf(\mathcal{O}_{\C_{p}})$ with generic fiber $X$.
\begin{defn}
Define $\mathcal{B}^{s}(\mathcal{X})$ to be the category of vector bundles $\mathcal{E}$ on $\mathcal{X}$ for which $\mathcal{E}\otimes\overline{\F}_{p}$ is a numerically flat vector bundle.\\
Similarly, we define $\mathcal{B}^{s}(X)$ to be the category of vector bundles on $X$ for which there exists an integral formal model with numerically flat reduction.
\end{defn}
We remark the following (compare with \cite[\S 9]{DW3}):
\begin{lem}
The categories $\mathcal{B}^{s}(X)$ and $\mathcal{B}^{s}(\mathcal{X})$ are closed under tensor products, extensions, duals, internal homs and exterior powers.
\end{lem}
\begin{proof}
The claim for $\mathcal{B}^{s}(\mathcal{X})$ follows from the analogous statement for numerically flat vector bundles on a proper scheme, which is well known (alternatively it can be deduced from Theorem \ref{thm31}).\\
For $\mathcal{B}^{s}(X)$ we just note that by the fundamental results of Raynaud, for any two formal models $\mathcal{X}, \mathcal{Y}$ of $X$, there exists an admissible blowup $\mathcal{X}'\to \mathcal{X}$, together with a morphism $\mathcal{X}'\to \mathcal{Y}$ (which can actually also be assumed to be an admissible blowup), which also induces an isomorphism on the generic fiber.\\
Using this the claim follows.
\end{proof}
Recall that we have a canonical projection $\mu:X_{pro\acute{e}t}\to \mathcal{X}_{Zar}$. For any $\mathcal{E}\in Vect(\mathcal{X})$ we again denote by $\mathcal{E}^{+}:=\mu^{-1}\mathcal{E}\otimes_{\mu^{-1}\mathcal{O}_{\mathcal{X}}}\mathcal{O}^{+}_{X}$ its pullback to the pro-\'etale site.
\begin{prop}\label{prop41}
For any $\mathcal{E}\in \mathcal{B}^{s}(\mathcal{X})$, the pullback to the pro-\'etale site $\mathcal{E}^{+}$ is contained in $\mathcal{B}^{p\acute{e}t}(\mathcal{O}_{X}^{+})$.
\end{prop}
\begin{proof}
By Theorem \ref{thm31} there exists a finite \'etale cover $f:Y_{0}\to \mathcal{X}\times Spec(\mathcal{O}_{\C_{p}}/p)$ such that $F_{Y_0}^{e*}f^{*}(\mathcal{E}/p)$ is trivial for some $e\geq 0$. We can lift $f$ to a finite \'etale cover $\mathcal{Y}\to \mathcal{X}$. Denote by $Y$ the generic fiber of $\mathcal{Y}$. Then $\mathcal{E}^{+}\vert_{Y}$ is $F^{e}$-trivial. So by Theorem \ref{tf} we have $\mathcal{E}^{+}\in \mathcal{B}^{p\acute{e}t}(\mathcal{O}^{+}_{X})$.
\end{proof}
So pullback to the pro-\'etale site gives a functor $\mathcal{B}^{s}(\mathcal{X})\to \mathcal{B}^{p\acute{e}t}(\mathcal{O}_{X}^{+})$. Using Theorem \ref{thm21} we get our main result:
\begin{thm}
The composition $\mathcal{B}^{s}(\mathcal{X})\xrightarrow{\mu^{*}} \mathcal{B}^{p\acute{e}t}(\mathcal{O}_{X}^{+})\xrightarrow{\rho_{\mathcal{O}}}Rep_{\pi_1(X)}(\mathcal{O}_{\C_p})$ is an exact functor of tensor categories
\begin{center}
$
DW:\mathcal{B}^{s}(\mathcal{X})\to Rep_{\pi_1(X)}(\mathcal{O}_{\C_p})
$
\end{center}
compatible with duals, internal homs and exterior products. Moreover, for any morphism $f:\mathcal{Y}\to \mathcal{X}$ of proper connected admissible formal schemes over $Spf(\mathcal{O}_{\C_p})$, with generic fiber $f_{\C_p}:Y\to X$, the following diagram commutes:
\begin{center}
$\xymatrix{
\mathcal{B}^{s}(\mathcal{X})\ar[d]^{f_{\C_p}^{*}}\ar[r]^-{DW}  & Rep_{\pi_1(X)}(\mathcal{O}_{\C_p})\ar[d]^{f_{\C_p}^{*}}\\
\mathcal{B}^{s}(\mathcal{Y})\ar[r]^-{DW} & Rep_{\pi_1(Y)}(\mathcal{O}_{\C_p}).
}$
\end{center}
\end{thm}
\begin{proof}
Everything follows from Theorem \ref{thm21} and Lemma \ref{lem11}.
\end{proof}
Using Theorem \ref{ff} we see that:
\begin{cor}
There is an exact tensor functor
\begin{center}
$
DW_{\Q}:\mathcal{B}^{s}(X)\to Rep_{\pi_1(X)}(\C_p)
$
\end{center}
which is compatible with duals, internal homs and exterior products. It is again compatible with pullback along any morphism $f:X\to Y$ of proper connected rigid analytic varieties.\\
If $X$ is seminormal, then $DW_{\Q}$ is fully faithful.
\end{cor}
In the language of \cite{Xu}, Proposition \ref{prop41} implies that all vector bundles with numerically flat reduction are Weil-Tate (compare Definition \ref{wt}). The case of curves has already been dealt with in  \cite[Corollaire 14.5]{Xu} (using the Faltings topos). All Weil-Tate vector bundles are semistable in the following sense (compare with \cite[Theorem 9.7]{DW3}):
\begin{prop}
Let $E$ be a vector bundle on $X$ such that $\lambda^{*}E\otimes \hat{\mathcal{O}}_{X}$ is trivialized by a profinite \'etale cover. Then $f^{*}E$ is semistable of degree $0$ for every morphism $f:C\to X$ where $C$ is s smooth projective curve.
\end{prop}
\begin{proof}
We have to show that any vector bundle $E$ on a smooth projective curve for which $\lambda^{*}E\otimes \hat{\mathcal{O}}_{X}$ is profinite \'etale trivializable, is semistable of degree $0$.  So assume that $X$ is a curve of genus $g$, and that $\tilde{Y}=\varprojlim{Y_{i}}\to X$ is a profinite \'etale cover trivializing $\lambda^{*}E\otimes \hat{\mathcal{O}}_{X}$. If $\tilde{Y}$ stabilizes, i.e. if $\lambda^{*}E\otimes \hat{\mathcal{O}}_{X}$ becomes trivial on some finite \'etale cover $f_{i}:Y_{i}\to X$, then the pullback $f_{i}^{*}E$ to $Y_{i}$ is also trivial (as $\lambda_{Y_{i}*}\hat{\mathcal{O}}_{Y_{i}}=\mathcal{O}_{(Y_{i})_{an}}$), so in particular semistable of degree $0$. But then $E$ is also semistable of degree $0$. So assume that $\tilde{Y}$ does not stabilize. We can assume that $deg(E)\geq 0$ (otherwise pass to the dual bundle). Let $r$ be the rank of $E$. Assume that $L\subset E$ is destabilizing. By passing to exterior powers we can assume that $L$ is a line bundle, and by passing to tensor powers we can assume that $deg(L)\geq g$. Then by Riemann-Roch $dim_{\C_p}\Gamma(Y_{i}, f_{i}^{*}L)\geq deg(f_{i})$, where $f_{i}$ is the finite \'etale map $Y_{i}\to X$. As $\tilde{Y}$ does not stabilize, $deg(f_{i})$ must grow to infinity. But $\varinjlim \Gamma(Y_{i}, f_{i}^{*}L)=\Gamma(\tilde{Y}, L)\subset \Gamma(\tilde{Y}, E)\subset \Gamma(\tilde{Y}, E\otimes \hat{\mathcal{O}}_{X})=\C_{p}^{r}$.\\
In the same way one may check that the degree of $E$ must be $0$.
\end{proof}
If $X$ is an algebraic variety the proposition says that $E$ is numerically flat. In general, if $X$ is not algebraic, one may expect that $E$ is semistable with respect to any polarization on the special fiber of a formal model in the sense of \cite{Li}. We do not pursue this question here.
\subsubsection{The case of line bundles}
We define $\hat{\C}_p:=\hat{\mathcal{O}}_{\C_p}[\frac{1}{p}]$, as sheaves on the pro-\'etale site, and say that a vector bundle $E$ is associated to a $\hat{\C}_p$-local system, if there exists a locally free $\hat{\C}_p$-module $\mathbb{L}$ such that $\mathbb{L}\otimes_{\hat{\C}_p} \hat{\mathcal{O}}_{X}\cong \hat{\mathcal{O}}_{X}\otimes \lambda^{*}E$.\\
The following proposition shows in particular that any torsion line bundle is associated to a local system.
\begin{prop}\label{proppow}
Let $L$ be a line bundle on $X$ such that $L^{\otimes n}$ is Weil-Tate. Then $L$ is associated to a $\hat{\C}_{p}$-local system $\mathbb{M}$.
\end{prop}
\begin{proof}
We use Kummer-theory on the pro-\'etale site. Let $\mathbb{L}$ be the $\hat{\mathcal{O}}_{\C_p}$-local system associated to $L^{\otimes n}$.
\begin{claim}
There is a commutative diagram
\begin{center}
$\xymatrix{
0\ar[r] & \mu_{n}\ar[r]\ar[d]^{=}& \hat{\mathcal{O}}_{\C_p}^{\times} \ar[r]^{x\to x^{n}}\ar[d] & \hat{\mathcal{O}}_{\C_p}^{\times} \ar[r]\ar[d] & 0\\
0\ar[r] & \mu_{n}\ar[r] & \hat{\mathcal{O}}_{X}^{\times}\ar[r]^{x\to x^{n}} & \hat{\mathcal{O}}_{X}^{\times} \ar[r] & 0
}$
\end{center}
of Kummer exact sequences.
\end{claim}
The commutativity of the diagram is clear. Moreover, so is exactness of the upper sequence. For the lower sequence, note that as usual the only non-trivial part is showing right exactness. This follows in the same way as exactness of the Artin-Schreier sequence (for the tilted structure sheaf) proved in the proof of \cite[Theorem 5.1]{Sch1}: If $U=\varprojlim Spa(A_{i}, A_{i}^{+})\in X_{pro\acute{e}t}$ is affinoid perfectoid with associated perfectoid space $\hat{U}=Spa(A, A^{+})$ (i.e. $(A, A^{+})$ is the completion of the direct limit of the $(A_{i}, A_{i}^{+})$), then $\hat{\mathcal{O}}_{X}(U)=A$, and so any equation $x^{n}-a=0$ for fixed $a\in \hat{\mathcal{O}}_{X}(U)$ can be solved by passing to a finite \'etale extension of $A$, which gives a finite \'etale map of affinoid perfectoid spaces $\tilde{V}\to \hat{U}$, say $\tilde{V}=Spa(B, B^{+})$. But now by \cite[Lemma 7.5]{Sch} any finite \'etale cover of $\hat{U}$ comes from a finite \'etale cover $V\to U$ in $X_{pro\acute{e}t}$ where then $V$ is affinoid perfectoid (as an object in $X_{pro\acute{e}t}$) with $\hat{V}=\tilde{V}$. Hence $\hat{\mathcal{O}}_{X}(V)=B$, so we can find an $n$-th root of $a$ by passing to the \'etale cover $V\to U$.\\
Now, taking pro-\'etale cohomology gives a commutative diagram
\begin{center}
$\xymatrix{
H^{1}(\mu_{n})\ar[r]\ar[d]^{=}& H^{1}(\hat{\mathcal{O}}_{\C_p}^{\times}) \ar[r]\ar[d] & H^{1}(\hat{\mathcal{O}}_{\C_p}^{\times}) \ar[r]\ar[d] & H^{2}(\mu_{n})\ar[d]^{=}\\
H^{1}(\mu_{n})\ar[r]& H^{1}(\hat{\mathcal{O}}_{X}^{\times}) \ar[r] & H^{1}(\hat{\mathcal{O}}_{X}^{\times}) \ar[r] & H^{2}(\mu_{n}),
}$
\end{center}
where the map $H^{1}(\hat{\mathcal{O}}_{\C_p}^{\times})\to H^{1}(\hat{\mathcal{O}}_{X}^{\times})$ is given by taking a rank $1$ local system $\mathbb{L}$ to $\hat{\mathcal{O}}_{X}\otimes \mathbb{L}$.\\
By some diagram chasing one gets from this that there exists an $\mathcal{O}_{\C_p}$-local system $\mathbb{H}$ such that $\mathbb{H}^{\otimes n}\cong \mathbb{L}$. From this one sees that 
\begin{center}$(\mathbb{H}\otimes \hat{\mathcal{O}}_{X})^{\otimes n}\cong \mathbb{L}\otimes \hat{\mathcal{O}}_{X}\cong (\lambda^{*}L \otimes \hat{\mathcal{O}}_{X})^{\otimes n}$. 
\end{center}
But then these sheaves become isomorphic on some finite \'etale Kummer cover $\pi:X'\to X$, so that $\lambda_{X'}^{*}(\pi^{*}L)\otimes \hat{\mathcal{O}}_{X'}=\pi^{*}_{pro\acute{e}t}(\lambda_{X}^{*}L\otimes \hat{\mathcal{O}}_{X})\cong \pi_{pro\acute{e}t}^{*}\mathbb{H}\otimes \hat{\mathcal{O}}_{X'}$. In particular $\pi^{*}L$ is Weil-Tate.\\
We conclude with the following lemma:
\begin{lem}
Let $\pi:X'\to X$ be a finite \'etale cover and let $E$ be a vector bundle on $X$ such that $\pi^{*}E$ is associated to an $\hat{\mathcal{O}}_{\C_{p}}$-local system $\mathbb{K}$. Then $E$ is associated to a $\hat{\C}_p$-local system.
\end{lem}
For this assume that $\pi$ is Galois with Galois group $G$. We then have a canonical Galois descent datum on $\pi_{pro\acute{e}t}^{*}(E\otimes \hat{\mathcal{O}}_{X})\cong \mathbb{K}\otimes \hat{\mathcal{O}}_{X'}$. This induces a descent datum on $\mathbb{K}[\frac{1}{p}]$, as $\mathbb{K}[\frac{1}{p}]\mapsto \mathbb{K}[\frac{1}{p}]\otimes \hat{\mathcal{O}}_{X'}$ is fully faithful (by Corollary \ref{full}), hence it descends to some $\mathbb{M}$ on $X$ and then, as the glueing datum is compatible with the one on $\pi_{pro\acute{e}t}^{*}(E\otimes \hat{\mathcal{O}}_{X})$, one has $\mathbb{M}\otimes \hat{\mathcal{O}}_{X}\cong E\otimes \hat{\mathcal{O}}_{X}$. 
\end{proof}
If $X$ is a proper algebraic variety over $\C_p$ there is a morphism of ringed spaces \begin{center}
$(X^{an}, \mathcal{O}_{X_{an}})\to (X, \mathcal{O}_{X})$ 
\end{center}
where $X^{an}$ denotes the analytificaction of $X$ in the sense of Huber. Then by non-archimedean GAGA the pullback along this induces an equivalence of categories between the categories of vector bundles on $X$ and $X^{an}$.\\
We denote by $Pic^{\tau}_{X}$ the torsion component of the identity of the Picard scheme of $X$, which parametrizes line bundles $L$ for which some power $L^{\otimes n}$ is algebraically equivalent to the trivial bundle.
\begin{prop}
Let $X$ be a proper geometrically connected geometrically normal algebraic variety over $K$ a finite extension of $\Q_p$. Then any $L \in Pic^{\tau}_{X_{\C_p}}(\C_p)$ is associated to a $\C_p$-local system.
\end{prop}
\begin{proof}
We may assume that $X$ has a $K$-rational point (otherwise pass to a finite extension of the base field). As $X$ is geometrically normal, $Pic^{0}_{X}$ is then an abelian variety (see \cite[Proposition 6.2.2]{Br}).
We now use similar arguments as in the proof of \cite[Proposition 10.2]{DW3}: 
By a theorem of Temkin (see \cite[Theorem 3.5.5]{Te}), after passing to a finite extension $K\subset K'$, we can assume that $X_{K'}$ admits a proper flat model $\mathcal{X}$ over $Spec(\mathcal{O}_{K'})$ with geometrically reduced fibers. Moreover, by \cite[Proposition 3.6]{DW3}, $\mathcal{X}$ is integral and $\mathcal{X}\to Spec(\mathcal{O}_{K'})$ is cohomologically flat in dimension $0$. Then $Pic^{\tau}_{\mathcal{X}}$ is representable by a separated group scheme over $Spec(\mathcal{O}_{K})$ by \cite[Corollaire (6.4.5)]{Ray}. Now $Pic^{0}_{X_{K'}}$ is an open sub-group scheme of $Pic^{\tau}_{X_{K'}}$ and hence $Pic^{0}_{X_{K'}}(\C_p)$ is a $p$-adically open subgroup of $Pic^{\tau}_{X_{K'}}(\C_p)$. Moreover $Pic^{\tau}_{\mathcal{X}}(\mathcal{O}_{\C_p})$ is $p$-adically open in $Pic^{\tau}_{X_{K'}}(\C_p)$. Hence the intersection $Pic^{\tau}_{\mathcal{X}}(\mathcal{O}_{\C_p})\cap Pic^{0}_{X_{K'}}(\C_p)$ is a $p$-adically open subgroup of $Pic^{0}_{X_{K'}}(\C_p)$. As $Pic^{0}_{X_{K'}}$ is an abelian variety we can then apply \cite[Theorem 4.1]{Col} to see that the quotient of $Pic^{0}_{X_{K'}}(\C_p)$ by $Pic^{\tau}_{\mathcal{X}}(\mathcal{O}_{\C_p})\cap Pic^{0}_{X_{K'}}(\C_p)$ is torsion. In particular for any line bundle $L\in Pic^{0}_{X_{K'}}(\C_p)$ there is some $n\geq 1$ such that $L^{\otimes n}$ admits an integral model which lies in $Pic^{\tau}_{\mathcal{X}}(\mathcal{O}_{\C_p})$. Hence also for any $L\in Pic^{\tau}_{X_{K'}}(\C_p)$ some power $L^{\otimes n}$ admits a model that lies in $Pic^{\tau}_{\mathcal{X}}(\mathcal{O}_{\C_p})$. But for any such model $\mathcal{M}$, the special fiber $\mathcal{M}_{\overline{\mathbb{F}}_{p}}$ is numerically flat. Now $\mathcal{M}$ induces a formal vector bundle $\hat{\mathcal{M}}$ on $\hat{\mathcal{X}}_{\C_p}$ where $\hat{\mathcal{X}}$ denotes the formal scheme obtained by completing $\mathcal{X}$ along its special fiber. Then the reduction of $\hat{\mathcal{M}}$ coincides with the reduction of $\mathcal{M}$, so is numerically flat. The generic fiber of $\hat{\mathcal{M}}$ is moreover given by $(L^{an})^{n}$, so that $(L^{an})^{n}$ is Weil-Tate. Thus, by Proposition \ref{proppow}, any $L\in Pic^{\tau}_{X_{K'}}(\C_p)$ is associated to a $\hat{\C}_p$ local system. 
\end{proof}
We recall that if $X$ is a projective scheme, then the line bundles in $Pic^{\tau}(X)$  are precisely the numerically flat line bundles.
\subsection{\'Etale parallel transport}
In this section we wish to show that the discussion in section 3.1 can be upgraded to construct \'etale parallel transport on pro-\'etale trivializable vector bundles. We will then compare our construction to the one from \cite{DW3}. This is in some sense close to the discussion in \cite[\S 8]{Xu}, where the results from \cite{DW1} for the curve case are recast in light of the Faltings topos. Note however that, in contrast to loc. cit., even though we need to pass through almost mathematics, our final statement (Theorem \ref{dwthm}) will be an honest isomorphism even at the integral level.\\
We will first recall the concept of \'etale parallel transport.\\
\begin{defn}\cite[\S 3]{DW1}
Let $X$ be a proper, connected rigid analytic variety over $\C_p$. The \'etale fundamental groupoid $\Pi_{1}(X)$ of $X$ is defined to be the category whose objects are given by $X(Spa(\C_{p}, \mathcal{O}_{\C_p}))$ and for any two points $x, y\in X(Spa(\C_p, \mathcal{O}_{\C_p}))$ we set $Mor(x, y)=Isom(F_{x}, F_{y})$. Here $F_{x}$ denotes the \'etale fiber functor with respect to the point $x$. $Mor(x, y)$ carries the profinite topology, making $\Pi_{1}(X)$ into a topological groupoid.
\end{defn}
\begin{rmk}\label{lequ}
Assume that $X$ is a finite type scheme over $Spec(\C_p)$, and denote by $X^{an}$ its analytification. By \cite[Theorem 3.1]{L1}, every finite \'etale cover of $X^{an}$ is algebraizable. In particular one gets an equivalence $\Pi_{1}(X)\cong \Pi_{1}(X^{an})$ of fundamental groupoids.
\end{rmk}
Let $Free_{r}(\mathcal{O}_{\C_p})$ (resp. $Free_{r}(\C_p)$) denote the topological groupoid of free rank $r$ modules over $\mathcal{O}_{\C_p}$ (resp. $\C_p$). Let $E$ be a vector bundle on $X$. We say that $E$ has \'etale parallel transport if the association $x\mapsto E_{x}$, can be extended to a functor $\Pi_{1}(X)\to Free_{r}(\C_p)$. Similarly, for a locally free $\mathcal{O}^{+}_{X_{an}}$-module $E^{+}$, we say that $E^{+}$ has \'etale parallel transport if $x\mapsto \Gamma(x^{*}E^{+})$ can be extended to a functor $\Pi_{1}(X)\to Free_{r}(\mathcal{O}_{\C_p})$.\\
Let $\mathcal{E}^{+}\in \mathcal{B}^{p\acute{e}t}(\mathcal{O}^{+}_{X})$, such that $\hat{\mathcal{E}}^{+}$ is trivial on $\tilde{Y}$. Let $r$ denote the rank of $\mathcal{E}^{+}$. We can define a functor 
\begin{center}
$
\alpha_{\hat{\mathcal{E}}^{+}}:\Pi_{1}(X)\to Free_{r}(\mathcal{O}_{\C_p})
$
\end{center}
of pro-\'etale parallel transport on $\hat{\mathcal{E}}^{+}$ as in \cite[\S 4]{DW3} (see also section 3): 
On objects $\alpha_{\hat{\mathcal{E}}^{+}}$ takes $x\in X(\C_p)$ to $\hat{\mathcal{E}}_{x}^{+}$. On morphisms, if we have an \'etale path $\gamma\in Mor_{\Pi(X)}(x, x')$, we can pick a $Spa(\C_p, \mathcal{O}_{\C_p})$-valued point $y$ of $\tilde{Y}$ lying over $x$, then $\gamma$ induces a point $y'$ over $x'$. Using triviality of $\mathcal{E}^{+}$ on $\tilde{Y}$ pulling back global sections along $y$, $y'$ will then give isomorphisms 
\begin{center}
$\hat{\mathcal{E}}_{x}^{+}\xleftarrow{y^{*}}\Gamma(\tilde{Y}, \hat{\mathcal{E}}^{+})\xrightarrow{y'^{*}} \hat{\mathcal{E}}_{x'}^{+}$.
\end{center}
and so we let $\gamma$ map to $y'^{*}\circ (y^{*})^{-1}$.\\
As in section 3.1 one can check that this is independent of the trivializing cover $\tilde{Y}$ and the chosen point $y$. Furthermore one can check in a similar fashion as in section 3.1 that this is a continuous functor of topological groupoids. Fixing a base point then gives back the representation from Theorem \ref{thm21}.\\
Now, if $E$ is a vector bundle on $X$ such that the pullback $\lambda^{*}E$ to the  pro-\'etale site lies in $\mathcal{B}^{p\acute{e}t}(\mathcal{O}_{X})$ we can also define \'etale parallel transport on $E$: For any point $x:Spa(\C_p, \mathcal{O}_{\C_p})\to X$ there is a canonical isomorphism 
\begin{center}
$\Gamma(x^{*}E)\cong \Gamma(x^{*}(\lambda^{*}E))\cong \Gamma(x^{*}(\lambda^{*}E\otimes \hat{\mathcal{O}}_{X}))$.
\end{center}
Hence, using \'etale parallel transport on $E\otimes \hat{\mathcal{O}}_{X}$, we get an isomorphism 
$E_{x}\xrightarrow{\cong}E_{y}$ for any \'etale path $x\mapsto y$. Moreover this construction is compatible with composition of \'etale paths.\\
The same construction also works for locally free $\mathcal{O}_{X_{an}}^{+}$-modules $E^{+}$ whose pullback to the pro-\'etale site lies in $\mathcal{B}^{p\acute{e}t}(\mathcal{O}_{X}^{+})$. We arrive at:
\begin{prop}
Let $E$ be a vector bundle on $X$ (resp. let $E^{+}$ be a locally free $\mathcal{O}^{+}_{X_{an}}$-module) for which the pullback to the pro-\'etale site $\lambda^{*}E$ (resp. $\lambda^{*}E^{+}$) lies in $\mathcal{B}^{p\acute{e}t}(\mathcal{O}_{X})$ (resp. $\mathcal{B}^{p\acute{e}t}(\mathcal{O}^{+}_{X})$). Then $E$ (resp. $E^{+}$) has \'etale parallel transport.\\
In particular we get functors
\begin{center}
$
\alpha:\mathcal{B}^{p\acute{e}t}(\mathcal{O}_{X_{an}})\to Rep_{\Pi_{1}(X)}(\C_{p})$\\
$\alpha_{\mathcal{O}_{\mathbb{C}_{p}}}:\mathcal{B}^{p\acute{e}t}(\mathcal{O}^{+}_{X_{an}})\to Rep_{\Pi_{1}(X)}(\mathcal{O}_{\mathbb{C}_{p}})
$.
\end{center}
\end{prop}
\begin{rmk}
Here $Rep_{\Pi_{1}(X)}(\mathcal{O}_{\mathbb{C}_{p}})$ denotes the category of continuous functors from $\Pi_{1}(X)$ to $Free_{r}(\mathcal{O}_{\mathbb{C}_{p}})$. A functor $F:\Pi_{1}(X)\to Free_{r}(\mathcal{O}_{\mathbb{C}_{p}})$ is called continuous if the induced maps on morphisms $Mor(x, x')\to Mor(F(x), F(x'))$ are continuous maps for all\\ $x, x'\in \Pi_{1}(X)$.
\end{rmk}
Note that whenever $E^{+}=\mathcal{O}^{+}_{X_{an}}\otimes_{sp^{-1}\mathcal{O}_{\mathcal{X}}}sp^{-1}(\mathcal{E})$ comes from an integral model $(\mathcal{X}, \mathcal{E})$ of $(X, E)$ the canonical isomorphisms $\Gamma(sp(x)^{*}\mathcal{E})\cong \Gamma(x^{*}E^{+})$ also allow us to define parallel transport on $\mathcal{E}$.\\ 
For the comparison with the Deninger-Werner construction we need to work modulo $p^{n}$. As we have less control here, we need to pass through the almost setting.\\
So denote by $Free_{r}((\mathcal{O}_{\C_p}/p^{n})_{*})$ the groupoid (endowed with the discrete topology) of free $(\mathcal{O}_{\C_p}/p^{n})_{*}$-modules of rank $r$, where $(\mathcal{O}_{\C_p}/p^{n})_{*}$ again denotes the almost elements. We can then similarly define a mod $p^{n}$ almost parallel transport
\begin{center}
$
\alpha^{a}_{n}(\mathcal{E}^{+}/p^{n}):\Pi_{1}(X)\to Free_{r}((\mathcal{O}_{\C_p}/p^{n})_{*})
$.
\end{center}
Now as in the proof of Proposition \ref{prop31} we see that $(\mathcal{O}_{\C_p}/p^{n})^{r}=\Gamma(\tilde{Y}, \hat{\mathcal{E}}^{+})/p^{n}\xhookrightarrow{} \Gamma(\tilde{Y}, \mathcal{E}^{+}/p^{n})_{*}$ realizes $\alpha(\mathcal{E})(\gamma)/p^{n}$ as a subobject of $\alpha^{a}_{n}(\mathcal{E}^{+}/p^{n})(\gamma)$ for any \'etale path. Here by subobject we mean that there is a commutative diagram
\begin{gather}
\begin{aligned}
\xymatrix@C+4pc{
\hat{\mathcal{E}}_{x}^{+}/p^{n}\ar[r]^{\alpha(\mathcal{E})(\gamma)/p^{n}}\ar@{^{(}->}[d] & \hat{\mathcal{E}}_{x'}^{+}/p^{n}\ar@{^{(}->}[d] \\
(\hat{\mathcal{E}}_{x}^{+}/p^{n})_{*}\ar[r]^-{\alpha^{a}_{n}(\mathcal{E}^{+}/p^{n})(\gamma)} & (\hat{\mathcal{E}}_{x'}^{+}/p^{n})_{*}
}
\end{aligned}
\label{diag}
\end{gather}
for each \'etale path $\gamma$ from $x$ to $x'$, and this association is compatible with composition of paths.\\
Assume now that $X$ is a proper smooth algebraic variety over $\overline{\Q}_{p}$ with a flat proper integral model $\mathcal{X}$ over $\overline{\Z}_{p}$. Assume further, that we have a vector bundle $\mathcal{E}$ on $\mathcal{X}_{\mathcal{O}_{\C_p}}$ with numerically flat reduction. Then one of the main results in \cite{DW3} is the following:
\begin{prop}\cite[Theorem 7.1]{DW3}\label{propdw}
Fix $n\geq 1$. Then there exists an open cover $\{U_{i}\}$ of $X$, such that for every $i$ there is a proper surjective map $f_{i}:\mathcal{Y}_{i}\to \mathcal{X}$ which is finite \'etale over $U_{i}$ and is such that $f_{i}^{*}\mathcal{E}$ is trivial mod $p^{n}$.
\end{prop}
One can then moreover assume that $\mathcal{Y}_{i}$ is a good model in the sense of \cite[Definition 3.5]{DW3}. Using this result they construct a parallel transport functor 
\begin{center}
$\alpha^{DW}_{n, i}(\mathcal{E}):\Pi_{1}(U_{i})\to Free_{r}(\mathcal{O}_{\C_p}/p^{n})$
\end{center} 
for every $i$, which is then shown to glue to a functor $\alpha^{DW}_{n}(\mathcal{E})$ from $\Pi_{1}(X)$, which is again functorial in $\mathcal{E}$.\\
The construction of $\alpha^{DW}_{n, i}(\mathcal{E})$ is of course as above: For any \'etale path $\gamma$ from $x$ to $x'\in U_{i}(\C_{p})$ one gets an isomorphism 
\begin{center}
$
\Gamma(\overline{x}^{*}\mathcal{E}/p^{n})=\Gamma(\overline{y}^{*}f_{i}^{*}(\mathcal{E}/p^{n})\xrightarrow{\cong} \Gamma(f_{i}^{*}(\mathcal{E}/p^{n}))\xrightarrow{\cong}\Gamma((\overline{\gamma y})^{*}f_{i}^{*}(\mathcal{E}/p^{n}))=\Gamma(\overline{x}'^{*}\mathcal{E}/p^{n})
$
\end{center}
where $\overline{x},\cdots$ denotes the specialization with respect to the integral model $\mathcal{X}$ (resp. $\mathcal{Y}_{i}$).
Taking the projective limit one gets $\alpha^{DW}_{\mathcal{O}_{\C_p}}(\mathcal{E}):\Pi_{1}(X)\to Free_{r}(\mathcal{O}_{\C_p})$.
Now let $\hat{\mathcal{X}}$ be the admissible formal scheme obtained by completing $\mathcal{X}_{\mathcal{O}_{\C_p}}$ along its special fiber. We denote by $X^{an}$ the adic space generic fiber of $\hat{\mathcal{X}}$. $X^{an}$ then coincides with the analytification of $X_{\C_{p}}$. We have $\Pi_{1}(X)\cong \Pi_{1}(X^{an})$ by Remark \ref{lequ}. Let $\mathcal{E}\in \mathcal{B}^{s}(\mathcal{X}_{\mathcal{O}_{\C_p}})$. Pulling back $\mathcal{E}$ to $\hat{\mathcal{X}}$ gives an object $\tilde{\mathcal{E}}$ in $\mathcal{B}^{s}(\hat{\mathcal{X}})$.\\
Pulling back $\tilde{\mathcal{E}}$ further to the pro-\'etale site gives an object $\mathcal{E}^{+}$ of $\mathcal{B}^{p\acute{e}t}(\mathcal{O}^{+}_{X^{an}})$.\\
Denote by $\tilde{\alpha}_{\mathcal{O}_{\C_p}}$ the composition
\begin{center}
$\tilde{\alpha}_{\mathcal{O}_{\C_p}}:\mathcal{B}^{s}(\mathcal{X}_{\mathcal{O}_{\C_p}})\to\mathcal{B}^{s}(\hat{\mathcal{X}})\xrightarrow{sp^{*}}\mathcal{B}^{p\acute{e}t}(\mathcal{O}^{+}_{(X^{an})_{an}})\xrightarrow{\alpha_{\mathcal{O}_{\C_p}}}Rep_{\Pi_{1}(X)}(\mathcal{O}_{\mathbb{C}_{p}})$
\end{center}
where $sp^{*}(\mathcal{E})=sp^{-1}(\mathcal{E})\otimes \mathcal{O}^{+}_{(X^{an})_{an}}$, and $sp:(X^{an})_{an}\to \hat{\mathcal{X}}_{Zar}$ denotes the specialization.
\begin{thm}\label{dwthm}
The functors $\tilde{\alpha}_{\mathcal{O}_{\C_p}}$ and $\alpha^{DW}_{\mathcal{O}_{\C_p}}$ are naturally isomorphic.
\end{thm}
\begin{proof}
We will show that $\tilde{\alpha}_{\mathcal{O}_{\C_p}}(\mathcal{E})$ is isomorphic to $\alpha^{DW}_{\mathcal{O}_{\C_p}}(\mathcal{E})$. Functoriality in $\mathcal{E}$ will be left to the reader.\\
First fix $n\geq 1$, and let $\{U_{i} \}$ be an open cover of $X$ and $\{\mathcal{Y}_{i}\}$ the associated proper, surjective covers trivializing $\mathcal{E}/p^{n}$ as in Proposition \ref{propdw}. Denote by $Y_{i}$ the generic fiber of $\mathcal{Y}_{i}$. Let $Z_{n}\to X_{\C_p}$ be a finite \'etale cover trivializing $\mathcal{E}^{+}/p^{n}$. Let $\mathcal{Z}_{n}\to \mathcal{X}_{\mathcal{O}_{\C_p}}$ be a good model for $Z_{n}$ in the sense of \cite[Definition 3.5]{DW3}. The cover $f_{i}:\mathcal{Z}_{n}\times_{\mathcal{X}_{\mathcal{O}_{\C_p}}}\mathcal{Y}_{i \mathcal{O}_{\C_p}}\to \mathcal{X}_{\mathcal{O}_{\C_p}}$ is still a trivializing cover for $\mathcal{E}/p^{n}$ and finite \'etale over $U_{i}$. We can moreover assume that it is connected; otherwise one uses \cite[Lemma 3.12]{DW3} to find a connected cover dominating $\mathcal{Z}_{n}\times_{\mathcal{X}_{\mathcal{O}_{\C_p}}}\mathcal{Y}_{i \mathcal{O}_{\C_p}}$ which is still finite \'etale over $U_{i}$. Denote by $f_{i}^{an}$ the analytification of $f_{i}\otimes \C_p$. Then $f_{i}^{an}$ trivializes $\mathcal{E}^{+}/p^{n}$ and $(\tilde{\alpha}_{\mathcal{O}_{\C_p}}/p^{n})\vert_{U_{i}}$ can be realized on $f_{i}^{an*}\mathcal{E}^{+}/p^{n}$. Let $x, x'\in U_{i}(\C_{p})$ and let $\gamma\in Mor(x, x')$ be an \'etale path. Consider the following commutative diagram of $\mathcal{O}_{\C_p}/p^{n}$-modules: 
\begin{center}
$\xymatrix@C=1em{
(\mathcal{E}/p^{n})_{\overline{x}} \ar[r]^-{=}\ar@{^{(}->}[d] & \Gamma(\overline{y}^{*}f_{i}^{*}(\mathcal{E}/p^{n}))\ar[r]^-{\cong}\ar@{^{(}->}[d] & \Gamma(f_{i}^{*}(\mathcal{E}/p^{n}))\ar[r]^-{\cong}\ar@{^{(}->}[d]& \Gamma((\overline{\gamma y})^{*}f_{i}^{*}(\mathcal{E}/p^{n}))\ar@{^{(}->}[d]\ar[r]^-{=} & (\mathcal{E}/p^{n})_{\overline{x}'} \ar@{^{(}->}[d] 
\\
(\mathcal{E}^{+}_{x}/p^{n})_{*} \ar[r]^-{=} & \Gamma(y^{*}f^{an*}_{i}(\mathcal{E}^{+}/p^{n}))_{*}\ar[r]^-{\cong} & \Gamma(f^{an*}_{i}(\mathcal{E}^{+}/p^{n}))_{*}\ar[r]^-{\cong}& \Gamma((\gamma y)^{*}f^{an*}_{i}(\mathcal{E}^{+}/p^{n}))_{*}\ar[r]^-{=} & (\mathcal{E}^{+}_{x'}/p^{n})_{*}
}$
\end{center}
where the upper row is $(\alpha^{DW}_{n, i}(\mathcal{E})(\gamma))$ and the lower row is $\alpha^{a}_{n}(\mathcal{E}^{+}/p^{n})(\gamma)$ and $y$ is a point of $(\mathcal{Z}_{n}\times_{\mathcal{X}}\mathcal{Y}_{i})_{\C_p}$ above $x$, and $\overline{x}\in \mathcal{X}(\mathcal{O}_{\C_p})$ is the specialization of $x$.\\
For the construction of the vertical maps note that if $\mathcal{Z}$ is a proper scheme with associated formal scheme $\hat{\mathcal{Z}}$, one has a composition of morphisms of ringed sites
\begin{center}
$(Z_{pro\acute{e}t}, \mathcal{O}^{+}_{Z})\to (\hat{\mathcal{Z}}_{Zar}, \mathcal{O}_{\hat{\mathcal{Z}}})\to (\mathcal{Z}_{Zar}, \mathcal{O}_{\mathcal{Z}})$.
\end{center}
The vertical maps are then given by pulling back global sections along this and embedding into almost elements at the end. As everything is functorial, the diagram commutes. All vertical arrows are canonical almost isomorphisms.\\
More precisely, going through all identifications, one checks that this takes $\alpha^{DW}_{n, i}(\mathcal{E})(\gamma)$ isomorphically to $(\tilde{\alpha}/p^{n})(\mathcal{E})(\gamma)\xhookrightarrow{} \alpha^{a}_{n}(\mathcal{E}^{+}/p^{n})(\gamma)$ (see diagram \eqref{diag}): Indeed, the map on the fiber $(\mathcal{E}/p^{n})_{\overline{x}}\xhookrightarrow{} (\mathcal{E}^{+}_{x}/p^{n})_{*}$ factors of course through $(\mathcal{E}^{+}_{x}/p^{n})\xhookrightarrow{} (\mathcal{E}^{+}_{x}/p^{n})_{*}$. Also, everything is compatible with the composition of paths.\\
From this we get an isomorphism $(\alpha^{DW}_{n, i}(\mathcal{E})(\gamma))\cong (\tilde{\alpha}/p^{n})\vert_{U_i}$. But then, using the Seifert-van-Kampen result from \cite[Theorem 4.1]{DW3}, we see that $(\alpha^{DW}_{n})\cong (\tilde{\alpha}/p^{n})$ are isomorphic.\\
Passing to the $p$-adic completion we get the desired result.
\end{proof}
\subsection{The Hodge-Tate spectral sequence}
Keep assuming that $X$ is proper smooth over $Spa(\C_{p}, \mathcal{O}_{\C_p})$.
\begin{thm}\cite[Theorem 3.20]{Sch2}\label{thm46}
Let $E$ be a vector bundle on $X$ associated to an $\hat{\mathcal{O}}_{\C_p}$-local system $\mathbb{L}$.
Then there is a Hodge-Tate spectral sequence.
\begin{center}
$
E_{2}^{ij}=H^{i}(X, E\otimes \Omega_{X}^{j}(-j)) \Longrightarrow H^{i+j}(X_{pro\acute{e}t}, \mathbb{L})[\frac{1}{p}]
$
\end{center}
\end{thm}
Recall that $\Omega_{X}^{j}(-j):=\Omega_{X}^{j}\otimes_{\hat{\Z}_p}\hat{\Z}_{p}(-1)^{\otimes j}$, where $\hat{\Z}_{p}(1):=\varprojlim_{n}\mu_{p^{n}}$ as pro-\'etale sheaves, and $\hat{\Z}_{p}(-1)$ denotes the dual of $\hat{\Z}_{p}(1)$.
\begin{proof}
By assumption $\hat{E}=\hat{\mathcal{O}}_{X}\otimes \mathbb{L}$. Let again $\nu:X_{pro\acute{e}t}\to X_{\acute{e}t}$ denote the canonical projection. Then the Cartan-Leray spectral sequence reads
\begin{center}
$
H^{i}(X_{\acute{e}t},  R^{j}\nu_{*}(\hat{\mathcal{O}}_{X}\otimes \mathbb{L}))\Longrightarrow H^{i+j}(X_{pro\acute{e}t}, \hat{\mathcal{O}}_{X}\otimes \mathbb{L})
$.
\end{center}
By Theorem \ref{ciso} we have $H^{i+j}(X_{pro\acute{e}t}, \hat{\mathcal{O}}_{X}\otimes \mathbb{L})\cong H^{i+j}(X_{pro\acute{e}t}, \mathbb{L})[\frac{1}{p}]$.\\
But now, by \cite[Proposition 3.23]{Sch2} there is an isomorphism $R^{j}\nu_{*}\hat{\mathcal{O}}_{X}\cong \Omega_{X_{\acute{e}t}}^{j}(-j)$, and hence
\begin{center}
$
 R^{j}\nu_{*}(\hat{\mathcal{O}}_{X}\otimes \mathbb{L})=R^{j}\nu_{*}(\nu^{*}E\otimes_{\mathcal{O}_{X}}\hat{\mathcal{O}}_{X})\cong E\otimes R^{j}\nu_{*}\hat{\mathcal{O}}_{X}\cong E\otimes \Omega_{X_{\acute{e}t}}^{j}(-j)
 $
\end{center}
by the projection formula.
\end{proof}
As usual one does not get a canonical splitting of this spectral sequence in general. However one does have the following:
\begin{prop}
Assume that $X$ is a proper smooth rigid analytic space over $K$, where $K/\Q_p$ is a finite extension. Let further $E$ be a vector bundle on $X$ such that $E_{\hat{\bar{K}}}$ is associated to an $\hat{\mathcal{O}}_{\C_p}$-local system $\mathbb{L}$. Then the Hodge-Tate spectral sequence degenerates canonically at $E_{2}$.
\end{prop}
\begin{proof}
Now there is a semi-linear $G_{K}:=Gal(\bar{K}/K)$-action on the cohomology groups in Theorem \ref{thm46}, and in particular the differentials in the Cartan-Leray spectral sequence will be invariant under this action. But then, as 
\begin{center}
$H^{i}(X_{\hat{\bar{K}}}, E_{\hat{\bar{K}}}\otimes \Omega^{j}_{X_{\hat{\bar{K}}}}(-j))=H^{i}(X, E\otimes \Omega^{j}_{X})\otimes_{K} \C_p(-j)$
\end{center}
by base change for cohomology, one gets that all differentials are zero, as\\$Hom_{G_{K}}(\C_{p}(-j), \C_p(-j'))=0$ for $j\neq j'$ by Tate's theorem.
\end{proof}
In the general case (i.e. when $X$ is not defined over a discretely valued field) Guo has shown that any lifting of $X$ to $B_{dR}^{+}/\xi^{2} $ provides a non-canonical splitting 
\begin{center}
$R\nu_{*} \hat{\mathcal{O}}_{X}\cong \bigoplus_{i=0}^{dimX} \Omega_{X_{\acute{e}t}}^{i}(-i)[-i]$
\end{center}
in the derived category $D(\mathcal{O}_{X_{\acute{e}t}})$ (see \cite[Theorem 7.2.5]{guo}). The fact that any proper smooth $X$ admits such a lift can be shown by a spreading out argument (see \cite[Corollary 13.16]{BMS} - see also \cite[\S 7.4]{guo})
One then gets the following:
\begin{prop}
Let $E$ be a Weil-Tate vector bundle on $X$ associated to the $\hat{\mathcal{O}}_{\C_p}$-local system $\mathbb{L}$. Then the Hodge-Tate spectral sequence for $E$ degenerates (non-canonically).
\end{prop}
\begin{proof}
Using the projection formula, for any lift of $X$ to $B_{dR}^{+}/\xi^{2}$, one gets from \cite[Theorem 7.2.5]{guo}
\begin{center}
$R\nu_{*}(\hat{\mathcal{O}}_{X}\otimes_{\mathcal{O}_{X}}\nu^{*}E)\cong R\nu_{*}\hat{\mathcal{O}}_{X}\otimes_{\mathcal{O}_{X_{\acute{e}t}}} E\cong \bigoplus_{i=0}^{dimX} \Omega_{X_{\acute{e}t}}^{i}(-i)[-i]\otimes E$
\end{center}
as objects in the derived category. Remark here that the tensor products (and pullback along $\nu^{*}$) are derived, as $E$ is finite locally free.\\
Now the cohomology group 
\begin{center}
$H^{n}(X_{\acute{e}t}, \mathbb{L})[\frac{1}{p}]\cong H^{n}(X_{pro\acute{e}t}, \hat{\mathcal{O}}_{X}\otimes \mathbb{L})$
\end{center}
is given as the $n$-th cohomology group of the complex 
\begin{center}
$R\Gamma(X_{\acute{e}t}, R\nu_{*}(\hat{\mathcal{O}}_{X}\otimes \mathbb{L}))=R\Gamma(X_{\acute{e}t}, R\nu_{*}(\hat{\mathcal{O}}_{X}\otimes_{\mathcal{O}_{X}}\nu^{*}E))$.
\end{center}
Plugging in the first isomorphism above yields an isomorphism 
\begin{center}
$H^{n}(X_{\acute{e}t}, R\nu_{*}(\hat{\mathcal{O}}_{X}\otimes_{\mathcal{O}_{X}}\nu^{*}E))\cong \bigoplus_{j=0}^{dimX}H^{n-j}(X_{\acute{e}t}, \Omega_{X_{\acute{e}t}}^{j}(-j)\otimes E)$.
\end{center}
\end{proof}
For the constant local system, consider the map 
$\alpha:H^{1}(X, \mathcal{O}_{{X}_{\hat{\bar{K}}}})\to H_{\acute{e}t}^{1}(X, \Z_p)\otimes \hat{\bar{K}}$ from the Hodge-Tate decomposition. In \cite{DW2} Deninger and Werner showed that, if $X=A$ is an abelian variety with good reduction, there is a commutative diagram
\begin{center}
$\xymatrix{
H^{1}(A, \mathcal{O}_{{A}_{\hat{\bar{K}}}}) \ar[d]^{\cong}\ar[r]^{\alpha}  & H_{\acute{e}t}^{1}(A, \Z_p)\otimes \hat{\bar{K}}\ar[d]^{\cong}\\
Ext^{1}(\mathcal{O}_{{A}_{\hat{\bar{K}}}}, \mathcal{O}_{{A}_{\hat{\bar{K}}}})\ar[r]^{DW} & Ext^{1}(\hat{\C}_p, \hat{\C}_p)
}$
\end{center}
where the map below means applying the Deninger-Werner functor to a unipotent rank $2$ vector bundle, which gives a unipotent rank $2$ local system.\\
We can show that this generalizes to arbitrary extensions on any proper smooth rigid analytic variety. Namely, let $X$ be any proper smooth rigid analytic variety over $Spa(\C_p, \mathcal{O}_{\C_p})$. Then for any $E\in \mathcal{B}^{s}(X)$ the Hodge-Tate spectral sequence gives an injection \\$H^{1}(X, E)\to H^{1}_{\acute{e}t}(X, \mathbb{L}_{E})$. We then get the following
\begin{prop}
For any $E, E'\in \mathcal{B}^{s}(X)$ there is a commutative diagram
\begin{center}
$\xymatrix{
H^{1}(X, \mathcal{H}om(E, E')) \ar[d]^{\cong}\ar[r]^{\alpha}  & H_{\acute{e}t}^{1}(X, \mathcal{H}om(\mathbb{L}_{E}, \mathbb{L}_{E'}))\ar[d]^{\cong}\\
Ext^{1}(E, E')\ar[r]^{DW} & Ext^{1}(\mathbb{L}_{E}, \mathbb{L}_{E'})
}$
\end{center}
where for any $F\in \mathcal{B}^{s}(X)$ we denote by $\mathbb{L}_{F}$ the local system obtained by applying the functor $DW$, and $\alpha$ denotes the map coming from the Hodge-Tate spectral sequence.
\end{prop}
\begin{proof}
The map $\alpha$ is given by the composition 
\begin{center}
$
H^{1}(X, \mathcal{H}om(E, E'))\xrightarrow{\beta}H^{1}(X, \mathcal{H}om(E, E')\otimes \hat{\mathcal{O}}_{X})\xrightarrow{com} H_{\acute{e}t}^{1}(X, \mathcal{H}om(\mathbb{L}_{E}, \mathbb{L}_{E'}))$
\end{center}
where $com$ is the comparison isomorphism from Theorem \ref{ciso} and $\beta$ is simply given as the map on $H^{1}$ associated to the injection of pro-\'etale sheaves
\begin{center}
$\nu^{*}\nu_{*}(\hat{\mathcal{O}}_{X}\otimes \mathcal{H}om(\mathbb{L}_{E}, \mathbb{L}_{E'}))\xhookrightarrow{} \hat{\mathcal{O}}_{X}\otimes \mathcal{H}om(\mathbb{L}_{E}, \mathbb{L}_{E'})$.
\end{center}
Now take an extension 
\begin{center}
$
e=(0\to E'\to E\to E''\to 0)
$
\end{center}
of vector bundles on $X$. By Lemma \ref{lem24} the functors $DW(-)\otimes \hat{\mathcal{O}}_{X}$ and $\nu^{*}(-)\otimes \hat{\mathcal{O}}_{X}$ are canonically isomorphic. So one only has to check that the map $\beta$ takes the extension $e$ to
\begin{center}
$
0\to \nu^{*}E'\otimes \hat{\mathcal{O}}_{X}\to \nu^{*}E\otimes \hat{\mathcal{O}}_{X}\to \nu^{*}E''\otimes \hat{\mathcal{O}}_{X}\to 0
$.
\end{center}
But this is clear.
\end{proof}

\bibliographystyle{plain}

\addcontentsline{toc}{section}{References}
\begin{thebibliography}{10}

\bibitem[AGT16]{AGT}
Ahmed Abbes, Michel Gros, and Takeshi Tsuji.
\newblock {\em The p-adic Simpson correspondence}, volume 193.
\newblock Annals of Mathematics Studies (Princeton University Press), 04 2016.

\bibitem[BMS18]{BMS}
Bhargav Bhatt, Matthew Morrow, and Peter Scholze.
\newblock Integral p-adic hodge theory.
\newblock {\em Publ. Math. Inst. Hautes \'Etudes Sci.}, 128(1):219--397,
  Nov 2018.

\bibitem[BS17]{BS}
Bhargav Bhatt and Peter Scholze.
\newblock Projectivity of the witt vector affine grassmannian.
\newblock {\em Invent. Math.}, 209(2):329--423, Aug 2017.

\bibitem[Br18]{leb}
Arthur-César Le Bras.
\newblock Overconvergent relative de {R}ham cohomology over the {F}argues-{F}ontaine curve.
\newblock {\em arXiv e-prints}, page arXiv:1801.00429, 2018. 

\bibitem[BL93]{BL}
Siegfried Bosch and Werner L\"utkebohmert.
\newblock Formal and rigid geometry. ii. flattening techniques.
\newblock {\em Math. Ann.}, 296(3):403--430, 1993.

\bibitem[BLR95]{BLR}
Siegfried Bosch, Werner L\"utkebohmert, and Michel Raynaud.
\newblock Formal and rigid geometry. iv. the reduced fibre theorem.
\newblock {\em Invent. Math.}, 119(2):361--398, 1995.

\bibitem[Br17]{Br}
Brion, Michel.
\newblock Some structure theorems for algebraic groups.
\newblock{\em Proc. Symp. Pure Math.}, 94:53--125, 2017.

\bibitem[Col89]{Col}
R. F. Coleman.
\newblock Reciprocity laws on curves.
\newblock {\em Compositio Math.}, 72(2):205–-235, 1989.

\bibitem[DPS94]{DPS}
J.-P. Demailly, T.~Peternell, and M.~Schneider.
\newblock Compact complex manifolds with numerically effective tangent bundles.
\newblock {\em J. Algebraic Geom. 3}, pages 295--345, 1994.

\bibitem[DW05a]{DW2}
Christopher Deninger and Annette Werner.
\newblock {\em Line Bundles and p-Adic Characters}, pages 101--131.
\newblock Birkh{\"a}user Boston, Boston, MA, 2005.

\bibitem[DW05b]{DW1}
Christopher Deninger and Annette Werner.
\newblock Vector bundles on $p$-adic curves and parallel transport.
\newblock {\em Ann. Sci. \'Ec. Norm. Sup\'er.}, Ser.
  4, 38(4):553--597, 2005.

\bibitem[DW20]{DW3}
Christopher {Deninger} and Annette {Werner}.
\newblock {Parallel transport for vector bundles on p-adic varieties}.
\newblock {\em J. Algebraic Geom.}, 29:1--52, 2020.

\bibitem[Fal05]{Fa}
Gerd Faltings.
\newblock A p-adic Simpson correspondence.
\newblock {\em Advances in Mathematics}, 198(2):847 -- 862, 2005.
\newblock Special volume in honor of Michael Artin: Part 2.

\bibitem[Gir71]{gir}
J.~Giraud.
\newblock {\em Cohomologie non ab{\'e}lienne}.
\newblock Grundlehren der mathematischen Wissenschaften 179. Springer, 1971.

\bibitem[Guo19]{guo}
Haoyang Guo.
\newblock {{H}odge-{T}ate decomposition for non-smooth spaces}.
\newblock {\em arXiv e-prints}, page arXiv:1909.09917, 2019.

\bibitem[HW19]{HW}
Marvin~Anas {Hahn} and Annette {Werner}.
\newblock {Strongly semistable reduction of syzygy bundles on plane curves}.
\newblock {\em arXiv e-prints}, page arXiv:1907.02466, Jul 2019.

\bibitem[Hub94]{Hub}
R~Huber.
\newblock A generalization of formal schemes and rigid analytic varieties.
\newblock {\em Math. Z.}, 217:513--551, 09 1994.

\bibitem[Hub96]{Hu}
R~Huber.
\newblock {\em {\'E}tale Cohomology of Rigid Analytic Varieties and Adic
  Spaces}.
\newblock Vieweg+Teubner Verlag, 1996.

\bibitem[HL10]{HL}
Daniel Huybrechts and Manfred Lehn.
\newblock {\em The Geometry of Moduli Spaces of Sheaves}.
\newblock Cambridge Mathematical Library. Cambridge University Press, 2
  edition, 2010.

\bibitem[Kat73]{Ka}
Nicholas~M. Katz.
\newblock p-adic properties of modular schemes and modular forms.
\newblock In Willem Kuijk and Jean-Pierre Serre, editors, {\em Modular
  Functions of One Variable III}, pages 69--190, Berlin, Heidelberg, 1973.
  Springer Berlin Heidelberg.

\bibitem[KL15]{KLiu1}
Kiran~S. Kedlaya and Ruochuan Liu.
\newblock Relative p-adic hodge theory: Foundations.
\newblock {\em Asterisque}, 371, 01 2015.

\bibitem[KL16]{KLiu}
Kiran~S. {Kedlaya} and Ruochuan {Liu}.
\newblock {Relative p-adic Hodge theory, II: Imperfect period rings}.
\newblock {\em arXiv e-prints}, page arXiv:1602.06899, Feb 2016.

\bibitem[Kle14]{Kl}
Steven Kleiman.
\newblock The Picard scheme.
\newblock {\em Math. Surveys Monogr.}, 123, 02 2014.

\bibitem[LS77]{LS}
Herbert Lange and Ulrich Stuhler.
\newblock Vektorb\"undel auf Kurven und Darstellungen der algebraischen
  Fundamentalgruppe.
\newblock {\em Math. Z.}, 156:73--84, 1977.

\bibitem[Lan04]{La1}
Adrian Langer.
\newblock Semistable sheaves in positive characteristic.
\newblock {\em Ann. of Math.}, 159:251--276, 01 2004.

\bibitem[Lan11]{La}
Adrian Langer.
\newblock On the S-fundamental group scheme.
\newblock {\em Ann. Inst. Fourier (Grenoble)}, 61(5):2077--2119, 2011.

\bibitem[Li17]{Li}
Shizhang {Li}.
\newblock {On rigid varieties with projective reduction}.
\newblock {\em arXiv e-prints}, page arXiv:1704.03109, Apr 2017.

\bibitem[LZ16]{LZ}
Ruochuan Liu and Xinwen Zhu.
\newblock Rigidity and a Riemann-Hilbert correspondence for p-adic local
  systems.
\newblock {\em Invent. Math.}, 207, 02 2016.

\bibitem[L\"ut93]{L1}
W.~L{\"u}tkebohmert.
\newblock Riemann's existence problem for a p-adic field.
\newblock {\em Invent. Math.}, 111(1):309--330, Dec 1993.

\bibitem[MR84]{MR}
V.~B. {Mehta} and A.~{Ramanathan}.
\newblock {Restriction of stable sheaves and representations of the fundamental
  group.}
\newblock {\em Invent. Math.}, 77:163-172, 1984.

\bibitem[NS65]{NS}
M.~S. Narasimhan and C.~S. Seshadri.
\newblock Stable and unitary vector bundles on a compact Riemann surface.
\newblock {\em Ann. of Math.}, 82(3):540--567, 1965.

\bibitem[Ray70]{Ray}
M. Raynaud.
\newblock Sp\'ecialisation du foncteur de Picard. 
\newblock{\em Publ. Math. Inst. Hautes \'Etudes Sci.}, (38):27–76, 1970.

\bibitem[Sa09]{Sa}
T. Saito.
\newblock Wild ramification and the characteristic cycle of an l-adic sheaf.
\newblock {\em J. Inst. Math. Jussieu}, 8(4):769–829, 2009.

\bibitem[Sch12]{Sch}
Peter Scholze.
\newblock Perfectoid spaces.
\newblock {\em Publ. Math. Inst. Hautes \'Etudes Sci.}, 116(1):245--313,
  Nov 2012.

\bibitem[Sch13a]{Sch1}
Peter Scholze.
\newblock $p$ -adic Hodge theory for rigid analytic varieties.
\newblock {\em Forum Math. Pi}, 1:e1, 2013.

\bibitem[Sch13b]{Sch2}
Peter Scholze.
\newblock Perfectoid spaces: A survey.
\newblock {\em Current Developments in Mathematics}, 2012, 03 2013.

\bibitem[Sch16]{Schc}
Peter Scholze.
\newblock p-adic Hodge theory for rigid analytic varieties – corrigendum.
\newblock {\em Forum Math. Pi}, 4:e6, 2016.

\bibitem[Sim92]{Sim}
Carlos~T. Simpson.
\newblock Higgs bundles and local systems.
\newblock {\em Publ. Math. Inst. Hautes \'Etudes Sci.}, 75(1):5--95, Dec 1992.

\bibitem[Tem10]{Te}
Michael Temkin.
\newblock Stable modifications of relative curves.
\newblock {\em J. Algebraic Geom.}, 19:603--677, 2010.


\bibitem[UY86]{UY}
K.~Uhlenbeck and S.~T. Yau.
\newblock On the existence of hermitian-Yang-Mills connections in stable vector
  bundles.
\newblock {\em Comm. Pure Appl. Math.},
  39(S1):S257--S293, 1986.

\bibitem[War17]{War}
E.~Warner.
\newblock Adic moduli spaces.
\newblock {\em Phd thesis, Stanford University}, 2017.

\bibitem[Xu17]{Xu}
Daxin Xu.
\newblock Transport parall\`ele et correspondance de Simpson $p$-adique.
\newblock {\em Forum Math. Sigma}, 5:e13, 2017.

\end{thebibliography}
%\printbibliography
Institut f\"ur Mathematik, 
    Goethe-Universit\"at Frankfurt, 
    60325 Frankfurt am Main, 
    Germany\\
    wuerthen@math.uni-frankfurt.de
\end{document}